\documentclass[11pt]{article}

\usepackage{amsmath,amsthm}

\usepackage{amssymb,latexsym}

\usepackage{enumerate}

\usepackage{amsmath}
\usepackage{amssymb}
\def\toP{{\scriptstyle\,\buildrel{\P}\over{\hbox to
0.6cm{\rightarrowfill}}\,}}
\newcommand{\R}{{\mathbb R}}

\newcommand{\BR}{{\mathbb R}}

\newcommand{\sgn}{\mbox{\rm sgn}}

\newcommand{\N}{{\mathbb N}}
\newcommand{\LL}{{\mathbb L}}

\newcommand{\MM}{{\cal M}}

\newcommand{\arrowp}{\mathop{\rightarrow}_{P}}
\newcommand{\Rd}{{{{\mathbb R}^d}}}

\newcommand{\FF}{{\cal F}}

\newcommand{\esssup}{\mathop{\mathrm{ess\,sup}}}

\newtheorem{theorem}{\bf Theorem}[subsection]
\newtheorem{proposition}[theorem]{\bf Proposition}
\newtheorem{lemma}[theorem]{\bf Lemma}
\newtheorem{corollary}[theorem]{\bf Corollary}

\theoremstyle{definition}
\newtheorem{definition}[theorem]{Definition}
\newtheorem{remark}[theorem]{\bf Remark}

\newcommand{\nsubsection}{\setcounter{equation}{0}\subsection}

\begin{document}
\title{Reflected BSDEs with regulated trajectories}
\author{Tomasz Klimsiak, Maurycy Rzymowski and Leszek S\l omi\'nski}
\date{}
\maketitle
\begin{abstract}
We consider reflected backward stochastic different equations with
optional barrier and so-called regulated trajectories, i.e
trajectories with left and right finite limits. We prove existence
and uniqueness results. We also show that the solution may be
approximated by a modified penalization method. Application to an
optimal stopping problem is given.
\end{abstract}
{\em MSC 2000 subject classifications:} primary 60H10; secondary
60G40.
\medskip\\
{\em Keywords:} Reflected backward stochastic differential
equation, processes with regulated trajectories, modified
penalization method, optimal stopping problem.

\footnotetext{T. Klimsiak: Institute of Mathematics, Polish
Academy of Sciences, \'Sniadeckich 8, 00-956 Warszawa, Poland, and
Faculty of Mathematics and Computer Science, Nicolaus Copernicus
University, Chopina 12/18, 87-100 Toru\'n, Poland. E-mail:
tomas@mat.umk.pl}

\footnotetext{M. Rzymowski and L. S\l omi\'nski: Faculty of
Mathematics and Computer Science, Nicolaus Copernicus University,
Chopina 12/18, 87-100 Toru\'n, Poland. E-mails:
maurycyrzymowski@mat.umk.pl, leszeks@mat.umk.pl}

\nsubsection{Introduction} \label{sec1}

In the present paper we consider reflected backward stochastic
differential equations (RBSDEs for short) with Brownian
filtration, one barrier and $L^p$-data, $p\in[1,2]$. The main
novelty is that we only assume that the barrier is optional. As a
consequence the solutions of these equations need not be
c\`adl\`ag, but are so-called regulated processes, i.e. processes
whose trajectories have left and right finite limits. Our
motivation for studying such general equations comes from the
optimal stopping theory (see \cite{EK,GIOOQ,KQ,KQC}).

Let  $B$ be a standard $d$-dimensional Brownian motion  and let
$\mathbb{F}=\{\FF_t,t\in [0,T]\}$ be the standard augmentation of
the natural filtration generated by $B$. Suppose we are given an
$\mathbb{F}$-optional process $L=\{L_t,t\in[0,T]\}$, an
$\mathbb{F}$-adapted locally bounded variation process
$V=\{V_t,t\in[0,T]\}$, an $\FF_T$-measurable random variable $\xi$
such that $\xi\ge L_T$ (the terminal value) and  a measurable
function $f:[0,T]\times\Omega\times\R\times\R^{d}\to\R$
(coefficient). In the paper we consider RBSDEs  with barrier $L$
of the form
\begin{equation}
\label{eq1.1} Y_t=\xi+\int^T_tf(s,Y_s,Z_s)\,ds +\int_t^T\,dK_s
+\int_t^T\,dV_s-\int^T_t  Z_s\,dB_s,\quad t\in[0,T].
\end{equation}
Roughly speaking, by a solution to (\ref{eq1.1}) we understand a
triple $(Y,Z,K)$ of $\mathbb{F}$-progressively measurable
processes such that (\ref{eq1.1}) is satisfied, $Y$ has regulated
trajectories,
\begin{equation}
\label{eq1.6} Y_t\ge L_t,\quad t\in[0,T],
\end{equation}
and $K$ is an increasing process such $K_0=0$ satisfying some
minimality condition (see (\ref{eq1.5}) below). In case $L$ is
c\`adl\`ag this condition reads
\begin{equation}
\label{eq1.4} \int_0^T(Y_{t-}-L_{t-})\,dK_t=0.
\end{equation}
An important known result is  (see \cite{HHO}) that  for
c\`adl\`ag barrier  the solution $(Y,Z,K)$ of
(\ref{eq1.1})--(\ref{eq1.4}) leads to the solution of the
following  optimal stopping problem
\begin{equation}
\label{eqi.1} Y_t=\mathrm{ess}\sup_{\tau\in\Gamma_t}
E\Big(\int_t^\tau f(s,Y_s,Z_s)\,ds+\int_t^\tau\,dV_s +L_{\tau}{\bf
1}_{\{\tau<T\}}+\xi{\bf 1}_{\{\tau=T\}}|{\cal F}_t\Big),
\end{equation}
where $\Gamma_t$ is the set of all $\mathbb{F}$-stopping times
takin values in $[t,T]$. In case $L$ is not c\`adl\`ag, the
problem of right formulation of the minimal condition is more
complicated. Of course,  the minimal condition must ensure
uniqueness of solutions under reasonable assumptions on $f$. On
the other hand, we want (\ref{eqi.1}) to be satisfied. In the
present paper, for optional barrier $L$, we propose the following
minimality condition for $K$:
\begin{equation}
\label{eq1.5} \int^T_0(Y_{s-}-\limsup_{u\uparrow s}L_u)\,dK^*_s
+\sum_{s<T}(Y_s-L_s)\Delta^+K_s=0,
\end{equation}
where $K^*$ is the c\`adl\`ag part of process $K$ and
$\Delta^+K_t=K_{t+}-K_t$ (i.e. $\Delta^+K_t$ is the right-side
jump of $K$). Under this condition $(Y,Z)$ satisfies
(\ref{eqi.1}). Note that if $L$ and $K$ are c\`adl\`ag, then
(\ref{eq1.5}) reduces to (\ref{eq1.4}).

The fundamental results on RBSDEs  with Brownian filtration, one
continuous barrier and $L^2$-data were obtained in \cite{EKPPQ}.
These results were generalized to equations with two continuous
barriers in \cite{CK,HH}. Equations with continuous barriers and
$L^p$-data with $p\in[1,2)$ were studied for instance in
\cite{EHW,HP,kl1,RS2}. In most papers devoted to RBSDEs with
possibly discontinuous barriers it is assumed that the barriers
are c\`adl\`ag (see, e.g., \cite{HHO,LX,LX1}  and the references
therein). In \cite{PX} (the case $p=2$) and in \cite{kl} (the case
$p\in[1,2]$) progressively measurable barriers are considered. In
these papers the minimality condition for $K$ differs from
(\ref{eq1.4}) and from (\ref{eq1.5}), and what is more important
here, the first component $Y$ of the solution of (\ref{eq1.1})
need not satisfy (\ref{eq1.6}), but satisfies weaker condition
saying that $Y_t\ge L_t$ for a.e. $t\in [0,T]$. A serious drawback
to the last condition is that it does not lead to (\ref{eqi.1}).
In fact, in case $f=0$ and $V=0$, the first component $Y$ of the
solution of (\ref{eq1.1}) defined in \cite{kl,PX} is the strong
envelope of $L$ (for the notion of strong envelope see \cite{SZ}).
It is worth noting, however,  that the definition of a solution of
(\ref{eq1.1})  adopted in \cite{kl,PX} is suitable for
applications to the obstacle problem for parabolic PDEs (see
\cite{KR:JEE}).

The our knowledge, the paper by Grigorova at al. \cite{GIOOQ} is
the only paper dealing with RBSDEs with barriers that are not
c\`adl\`ag, and whose solution satisfies (\ref{eq1.6}) and
(\ref{eqi.1}). In the present paper we prove existence and
uniqueness results for (\ref{eq1.1}) which generalize  the
corresponding results of \cite{GIOOQ} in several directions. First
of all, we impose no regularity assumptions on $L$ (in
\cite{GIOOQ} it is assumed that $L$ is left-limited and right
upper-semicontinuous). Secondly,  we consider the case of
$L^p$-data with $p\ge 1$ (in \cite{GIOOQ} only the case of $p=2$
is considered). As for the generator, we assume that it is
Lipschitz continuous with respect to $z$ and only continuous and
monotone with respect to $y$ (in \cite{GIOOQ} it is assumed that
$f$ is Lipschitz continuous with respect to $y$ and $z$). Let us
also stress that the proofs of our results are totally different
from those of \cite{GIOOQ}. Our main new idea is to reduce the
problem for optional barriers to the problem for c\`adl\`ag
barriers.

In Section \ref{sec4} we consider the problem of approximation of
solutions of (\ref{eq1.1}) by solutions of usual BSDEs (this
problem was not considered in \cite{GIOOQ}). We show that the
solution of (\ref{eq1.1}) is the increasing limit of the sequence
$\{Y^n\}$ of solutions of the following penalized BSDEs
\begin{align*}
Y^n_t&=\xi+\int^T_tf(s,Y^n_s,Z^n_s)\,ds +\int^T_t\,dV_s
-\int^T_t Z^n_s\,dB_s\\
&\qquad+n\int_t^T(Y^n_s-L_s)^-ds +\sum_{t\leq
\sigma_{n,i}<T}(Y^n_{\sigma_{n,i}+}
+\Delta^+V_{\sigma_{n,i}}-L_{\sigma_{n,i}})^-,\quad t\in[0,T]
\end{align*}
with specially defined arrays of stopping times
$\{\{\sigma_{n,i}\}\}$ exhausting  right-side jumps of $L$ and
$V$. If $L, V$ are c\`adl\`ag then the term involving the
right-side jumps vanishes and our penalization scheme reduces to
the usual penalization for BSDEs with c\`adl`ag trajectories.

\nsubsection{Preliminaries}

Recall that a function $y:[0,T]\to\Rd$ is called regulated if for
every $t\in[0,T)$  the limit $y_{t+}=\lim_{u\downarrow t}y_u$
exists, and for every $s\in(0,T]$ the limit
$y_{s-}=\lim_{u\uparrow s}y_u$ exists. For any regulated function
$y$ on $[0,T]$ we set $\Delta^{+}y_t=y_{t+}-y_t$ if $0\leq t<T$,
and  $\Delta^{-}y_s=y_s-y_{s-}$ if  $0<s\leq T$  with the
convention that $\Delta^{+}y_T$ $=\Delta^{-}y_0$ $=0$  and $\Delta
y_t=\Delta^+y_t+\Delta^-y_t$, $t\in[0,T]$. It is known that each
regulated function is bounded and has at most countably many
discontinuities (see, e.g., \cite[Chapter 2, Corollary 2.2]{dn}).

For $x\in\R^d$, $z\in\R^{d\times n}$ we set
$|x|^2=\sum^d_{i=1}|x_i|^2$, $\|z\|^2=\mbox{trace}(z^*z)$.
$\langle\cdot,\cdot\rangle$ denotes the usual scalar product in
$\R^d$  and $\sgn(x)={\bf 1}_{\{x\neq0\}}{x}/{|x|}$.

By $L^p$, $p>0$, we denote the space of random variables $X$ such
that $\|X\|_p\equiv E(|X|^p)^{1\vee1/p}<\infty$. By $\mathcal{S}$
we denote the set of all $\mathbb{F}$-progressively measurable
processes with regulated trajectories, and by $\mathcal{S}^p$,
$p>0$, the subset of  $Y\in\mathcal{S}$ such that $E\sup_{0\le
t\le T}|Y_t|^p<\infty$. $\mathcal{H}$ is the set of
$d$-dimensional $\mathbb{F}$-progressively measurable processes
$X$ such that
\[
P\Big(\int^T_0|X_t|^2\,dt<\infty\Big)=1,
\]
and $\mathcal{H}^p$, $p>0$, is the set of all $X\in\mathcal{H}$
such that $\|X\|_{\mathcal
{H}^p}\equiv\|(\int_0^T|X_s|^2\,ds)^{1/2}\|_p<+\infty$.

We say that an $\mathbb{F}$-progressively measurable process $X$
is of class (D) if the family $\{X_{\tau},\,\tau\in\Gamma\}$ is
uniformly integrable, where $\Gamma$ is the set of all
$\mathbb{F}$-stopping times taking values in $[0,T]$. We equip the
space of processes of class (D) with the norm
$\|X\|_D=\sup_{\tau\in\Gamma}E|X_{\tau}|$.

For $\tau\in\Gamma$, by $[[\tau]]$ we denote the set
$\{(\omega,t):\,\tau(\omega)=t\}$. A sequence
$\{\tau_k\}\subset\Gamma$ is called stationary if
\[
\forall{\omega\in\Omega}\quad
\exists{n\in\mathbb{N}}\quad\forall{k\ge n}\quad\tau_k(\omega)=T.
\]
$\mathcal{M}_{loc}$ (resp. $\mathcal{M}$) is the set of all
$\mathbb{F}$-martingales (resp. local martingales) $M$ such that
$M_0=0$. $\MM^p$, $p\ge1$, denotes the space of all $M\in \MM$
such that
\[
E([M]_T)^{p/2}<\infty,
\]
where $[M]$ stands for  the quadratic variation of $M$.

$\mathcal{V}$ (resp. $\mathcal{V}^+$) denotes the space of
$\mathbb{F}$-progressively measurable process of finite variation
(resp. increasing) such that $V_0=0$,  and $\mathcal{V}^p$ (resp.
$\mathcal{V}^{+,p}$), $p\ge1$,  is the set of processes
$V\in\mathcal{V}$ (resp. $V\in\mathcal{V}^+$) such that
$E|V|^p_T<\infty$, where $|V|_T$ denotes the total variation of
$V$ on $[0,T]$.   For $V\in \mathcal{V}$, by $V^*$ we denote the
c\`adl\`ag part of the process $V$, and by  $V^d$ its purely
jumping part consisting of right jumps, i.e.
\[
V^d_t=\sum_{s<t}\Delta^+V_s,\quad V^*_t=V_t-V^d_t,\quad t\in
[0,T].
\]

Let $V^1,V^2\in \mathcal{V}$. We write $dV^1\le dV^2$ if
$dV^{1,*}\le dV^{2,*}$ and $\Delta^+V^1\le\Delta^+V^2$ on $[0,T]$.

In the whole paper all relations between random variables hold
$P$-a.s. For process $X$, $Y$ we  write $X\le Y$ if $X_t\le Y_t$,
$t\in[0,T]$. For a given optional process  $L$ of class  (D) we
set
\[
\mathrm{Snell}(L)_t=\esssup_{\tau\in\Gamma_t}E( L_{\tau}|\mathcal{F}_t),
\]
where $\Gamma_t$ is the set of all stopping times taking values in
$[t,T]$. From \cite{EK} it follows that the process Snell$(L)$ is
the smallest supermartingale dominating the process $L$.

We will need the following assumptions.
\begin{enumerate}
\item[(H1)]There is $\lambda\ge0$ such that
$|f(t,y,z)-f(t,y,z')|\le\lambda|z-z'|$ for all $t\in[0,T]$,
$y\in\R$, $z,z'\in\Rd$,
\item[(H2)] there is $\mu\in\R$ such that
$(y-y')(f(t,y,z)-f(t,y',z))\leq\mu(y-y')^2$ for all $t\in[0,T]$,
$y,y'\in\R$, $z\in\Rd$.
\item[(H3)] $\xi,\, \int_0^T|f(r,0,0)|\,dr,\, |V|_T\in L^p$,
\item[(H4)] for every $(t,z)\in[0,T]\times\mathbb{R}^d$ the
mapping $\mathbb{R}\ni y\rightarrow f(t,y,z)$ is continuous,
\item[(H5)] $[0,T]\ni t\mapsto f(t,y,0)\in L^1(0,T)$ for every $y\in\mathbb{R}$,
\item[(H6)] there exists a process $X$ such that
$E\sup_{0\le t\le T}|X_t|^p<\infty$,
$X\in\mathcal{M}_{loc}+\mathcal{V}^p$, $X\ge L$ and $\int_0^T
f^-(s,X_s,0)\,ds\in L^p$,
\item[(H6*)] there exists a process $X$ of class (D) such that
$X\in\mathcal{M}_{loc}+\mathcal{V}^1$, $X\ge L$ and $\int^T_0
f^-(s,X_s,0)\,ds\in L^1$,
\item[(Z)] there exists a progressively measurable process $g$
and $\gamma\ge 0,\, \alpha\in [0,1)$ such that
\[
|f(t,y,z)-f(t,y,0)|\le\gamma(g_t+|y|+|z|)^\alpha,
\quad t\in [0,T], \,y\in\BR,\, z\in\BR^d.
\]
\end{enumerate}

\begin{definition}
We say that a pair $(Y,Z)$ of $\mathbb{F}$-progressively
measurable processes  is a solution of BSDE with right-hand side
$f+dV$ and terminal condition $\xi$ (BSDE($\xi$,$f+dV$) in short)
if
\begin{enumerate}[{\rm(a)}]
\item$(Y,Z)\in \mathcal{S}^p\times\mathcal{H}$ for some
$p>1$ or $Y$ is of class {\rm{(D)}} and $Z\in\mathcal{H}^q$ for
$q\in (0,1)$,
\item$\int^T_0|f(s,Y_s,Z_s)|\,ds<\infty$,
\item$Y_t=\xi+\int^T_t f(s,Y_s,Z_s)\,ds+\int_t^T\,dV_s
-\int^T_t Z_s\,dB_s$, $t\in[0,T]$.
\end{enumerate}
\end{definition}

Theorems \ref{tw1.k1} and \ref{tw1.k2} below were proved in
\cite[Section 4]{kl} in case  $V$ is c\`adl\`ag. In the general
case, i.e. if $V\in\mathcal{V}$, their proofs go without any
changes. The only difference is that we use It\^o's formula for
regulated processes (see Appendix) instead  the usual It\^o's
formula.
\begin{theorem}
\label{tw1.k1} Let $p>1$. If  $\mbox{\rm{(H1)--(H5)}}$ are
satisfied then there exists a unique solution $(Y,Z)$ of
\textnormal{BSDE}$(\xi$,$f+dV)$. Moreover, $Z\in\mathcal{H}^p$ and
$E(\int^T_0|f(s,Y_s,Z_s)|\,ds)^p<\infty$.
\end{theorem}

\begin{theorem}
\label{tw1.k2} Let $p=1$. If \textnormal{(H1)--(H5), (Z)} are
satisfied then there exists a unique solution $(Y,Z)$ of
\textnormal{BSDE($\xi$,$f+dV$)}. Moreover, $Y\in\mathcal{S}^q$ for
every $q\in(0,1)$ and $E\int^T_0|f(s,Y_s,Z_s)|\,ds <\infty$.
\end{theorem}

Now we recall the definition of a solution of the reflected BSDE
in the class of c\`adl\`ag processes and results about existence
and uniqueness. Theorems \ref{tw1.kl5} and \ref{tw1.kl6} below
were proved in \cite{kl}.

\begin{definition} Assume that  $L, V$ are  c\`adl\`ag
processes. We say that  a triple  $(Y,Z,K)$ of
$\mathbb{F}$-progressively measurable processes is a solution of
reflected BSDE with right-hand side $f+dV$, terminal condition
$\xi$ and lower barrier $L$ ({\rm RBSDE(}$\xi$,$f+dV$,$L${\rm)} in
short) if
\begin{enumerate}[{\rm(a)}]
\item $(Y,Z)\in \mathcal{S}^p\times\mathcal{H}$ for some
$p>1$ or $Y$ is of class {\rm (D)} and $Z\in\mathcal{H}^q$ for
$q\in (0,1)$,
\item $K\in\mathcal{V}^+$ is c\`adl\`ag, $Y_t\ge L_t$, $t\in [0,T]$, and
$\int^T_0(Y_{s-}-L_{s-})\,dK_s=0$,
\item $\int^T_0|f(s,Y_s,Z_s)|\,ds<\infty$,
\item $Y_t=\xi+\int^T_t f(s,Y_s,Z_s)\,ds+\int_t^T\,dV_s+\int^T_t\,dK_s
-\int^T_t Z_s\,dB_s$, $t\in[0,T]$.
\end{enumerate}
\end{definition}

\begin{theorem}
\label{tw1.kl5} Let $p>1$ and $\mbox{\rm{(H1)--(H6)}}$ be
satisfied. Then  there exists a unique solution $(Y,Z,K)$ of {\rm
RBSDE($\xi$,$f+dV$,$L$)}. Moreover, $(Y,Z,K)
\in\mathcal{S}^p\otimes\mathcal{H}^p\otimes\mathcal{V}^{+,p}$ and
$E(\int^T_0|f(s,Y_s,Z_s)|\,ds)^p<\infty$.
\end{theorem}

\begin{theorem}
\label{tw1.kl6} Let $p=1$ and \textnormal{(H1)--(H5), (H6*), (Z)}
be satisfied. Then there exists a unique solution  $(Y,Z,K)$ of
\textnormal{RBSDE($\xi$,$f$,$L$)}. Moreover, $Y$ is of class {\rm
(D)},
$(Y,Z,K)\in\mathcal{S}^q\otimes\mathcal{H}^q\otimes\mathcal{V}^{1,+}$
for $q\in(0,1)$ and $E\int^T_0|f(s,Y_s,Z_s)|\,ds<\infty$.
\end{theorem}

For  convenience of the reader we now formulate  counterparts of
\cite[Lemma 4.11]{kl} and \cite[Theorem 4.12]{kl} for regulated
processes.

\begin{lemma}\label{kl4.11}
Assume that \textnormal{(H1)--(H4)} hold. Let
$L^n,L\in\mathcal{V}$, $g_n,g,\bar{f}$ be progressively measurable
processes such that $\int_0^T|g_n(s)|\,ds,\, \int_0^T|g(s)|\,ds,\,
\int_0^T|\bar{f}(s)|\,ds\in L^1$, and let $(Y^n,Z^n),
(Y,Z)\in\mathcal{S}\otimes\mathcal{H}$ be such that $t\mapsto
f(t,Y^n_t,Z^n_t),t\mapsto f(t,Y_t,Z_t)\in L^1(0,T)$ and
\[
Y^n_t=Y^n_0-\int^t_0 g_n(s)\,ds-\int^t_0 f(s,Y^n_s,Z^n_s)\,ds
-\int^t_0\,dL^n_s+\int^t_0 Z^n_s\,dB_s,\quad t\in[0,T],
\]
\[
Y_t=Y_0-\int^t_0 g(s)\,ds-\int^t_0 \bar{f}(s)\,ds
-\int^t_0\,dL_s+\int^t_0 Z_s\,dB_s,\quad t\in[0,T].
\]
If
\begin{enumerate}[{\rm(a)}]
\item $E\sup_{n\ge 0}(L^n)^+_T+E\int^T_0|f(s,0,0)|\,ds<\infty$,
\item $\liminf_{n\rightarrow\infty}
(\int^{\tau}_{\sigma}(Y_s-Y^n_s)\,dL^{n,*}_s+\sum_{\sigma\le
s<\tau}(Y_s-Y^n_s)\Delta^+L^n_s)\ge 0$ for all
$\sigma,\tau\in\Gamma$ such that $\sigma\le\tau$,
\item there exists $C\in\mathcal{V}^{1,+}$ such that
$|\Delta^-(Y_t-Y^n_t)|\le|\Delta^-C_t|$, $t\in[0,T]$,
\item there exist processes $\underline{y},
\overline{y}\in\mathcal{V}^{1,+}+\mathcal{M}_{loc}$ of class
\mbox{\rm(D)} such that
\[
\overline{y}_t\le Y_t\le\underline{y}_t,\quad t\in[0,T], \quad
E\int^T_0 f^+(s,\overline{y}_s,0)\,ds+E\int^T_0
f^-(s,\underline{y}_s,0)\,ds<\infty,
\]
\item there exists $h\in L^1(\mathcal{F})$ such
that $|g_n(s)|\le h(s)$ for a.e. $s\in[0,T]$,
\item $Y^n_t\rightarrow Y_t$, $t\in[0,T]$,
\end{enumerate}
then
\[
Z^n\rightarrow Z,\quad\lambda\otimes P\mbox{-a.e.},
\quad\int^{T}_0
|f(s,Y^n_s,Z^n_s)-f(s,Y_s,Z_s)|\,ds\rightarrow0\quad
in\,\,\mbox{probability}\quad P
\]
and there exists a sequence $\{\tau_k\}\subset\Gamma$ such that
for all  $k\in\mathbb{N}$ and $p\in(0,2)$,
\begin{equation}
\label{eq2.1} E\int^{\tau_k}_0|Z^n_s-Z_s|^p\,ds\rightarrow0.
\end{equation}
If $\Delta^-C_t=0$, $t\in[0,T]$, then \mbox{\rm(\ref{eq2.1})} also
holds for $p=2$. If additionally  $g_n\rightarrow g$ weakly in
$L^1([0,T]\times\Omega)$ and $L^n_{\tau}\rightarrow L_{\tau}$
weakly in $L^1$ for every $\tau\in\Gamma$, then
$\bar{f}(s)=f(s,Y_s,Z_s)$ for a.e. $s\in[0,T]$.
\end{lemma}
\begin{proof}
It is enough to repeat step by step the the proof  \cite[Lemma
4.11]{kl} and use It\^o's formula for regulated processes (see
Appendix). The only difference is that inequality (4.16) in
\cite{kl} in our case takes the form
\begin{align*}
E\int^{\tau}_{\sigma}|Z_s-Z^n_s|^2\,ds
&\le E|Y_{\tau}-Y^n_{\tau}|^2+2E\int^{\tau}_{\sigma}
|Y_s-Y^n_s||f(s,Y_s,Z_s)-f(s,Y^n_s,Z^n_s)|\,ds\\
&\ +2\int^{\tau}_{\sigma}|Y_s-Y^n_s||g(s)-g_n(s)|\,ds
+2E\int^{\tau}_{\sigma}(Y_s-Y^n_s)\,d(L_s-L^n_s)^*\\
&\ +2E\sum_{\sigma\le s<\tau}(Y_s-Y^n_s)\Delta^+(L_s-L^n_s)
+E\sum_{\sigma\le s<\tau}|\Delta^-(L_s-L^n_s)|^2.
\end{align*}
\end{proof}

\begin{remark}
\label{uw1.kl1} In Lemma \ref{kl4.11} assumption (e) may be
replaced by the following  one:  there exists a stationary
sequence $\{\tau_k\}\subset \Gamma$ such that $\sup_{n\ge1}
E\int^{\tau_k}_0|g_n(s)|^2\,ds<\infty$ and the assertion of the
lemma holds. This follows from the fact that assumption (e) is
used in the proof of \cite[Lemma 4.11]{kl} only to show that
\cite[(4.15)]{kl} holds true, i.e. that
$\int^T_0|g(s)-g_n(s)||Y_s-Y^n_s|\,ds\rightarrow0$. But under the
new condition this follows from the inequality
\begin{align*}\int^T_0|g(s)-g_n(s)||Y_s-Y^n_s|\,ds
\le\Big(E\int^T_0|g(s)-g_n(s)|^2\,ds\Big)^{1/2}
\Big(E\int^T_0|Y_s-Y^n_s|^2\,ds\Big)^{1/2}.
\end{align*}
\end{remark}

\begin{theorem}\label{kl4.12}
Assume that  \textnormal{(H1)--(H4)} hold,
$(Y^n,Z^n)\in\mathcal{S}\otimes\mathcal{H}$, $A^n\in \mathcal{V},
K^n\in\mathcal{V}^+$, $t\mapsto f(t,Y^n_t,Z^n_t)\in L^1(0,T)$ and
\[
Y^n_t=Y^n_0-\int^t_0 g_n(s)\,ds-\int^t_0 f(s,Y^n_s,Z^n_s)\,ds
-\int^t_0\,dK^n_s+\int^t_0\,dA^n_s+\int^t_0 Z^n_s\,dB_s
\]
for $t\in[0,T]$. Moreover, assume that
\begin{enumerate}[{\rm (a)}]
\item $dA^n\le dA^{n+1}$, $n\in\mathbb{N}$,
$\sup_{n\ge 0}E|A^n|_T<\infty$,

\item $\liminf_{n\rightarrow\infty}
\Big(\int^{\tau}_{\sigma}(Y_s-Y^n_s)\,d(K^n_s-A^n_s)^*
+\sum_{\sigma\le s<\tau}(Y_s-Y^n_s)\Delta^+(K^n_s-A^n_s)\Big)\ge
0$ for any $\sigma,\tau\in\Gamma$ such that $\sigma\le\tau$,

\item there exists process $C\in\mathcal{V}^{1,+}$ such that
$\Delta^-K^n_t\le\Delta^-C_t$, $t\in[0,T]$,

\item there exist processes $\underline{y},\overline{y}
\in\mathcal{V}^{1,+}+\mathcal{M}_{loc}$ of class (D) such that
\[
E\int^T_0 f^+(s,\overline{y}_s,0)\,ds+E\int^T_0
f^-(s,\underline{y}_s,0)\,ds<\infty,\quad\overline{y}_t
\le Y^n_t\le\underline{y}_t,\quad t\in[0,T],
\]
\item $E\int^T_0|f(s,0,0)|\,ds<\infty$ and there
exists a progressively measurable process $h\in
L^1([0,T]\times\Omega)$ such that $|g_n(s)|\le h(s)$ for a.e.
$s\in[0,T]$,
\item $Y^n_t\nearrow Y_t$, $t\in[0,T]$.
\end{enumerate}
Then $Y\in\mathcal{S}$ and there exist $K\in\mathcal{V}^+$,
$A\in\mathcal{V}^{1}$, $Z\in\mathcal{H}$ and progressively
measurable process $g\in L^1([0,T]\times \Omega)$ such that
\[
Y_t=Y_0-\int^t_0 g(s)\,ds-\int^t_0 f(s,Y_s,Z_s)\,ds
-\int^t_0\,dK_s+\int^t_0\,dA_s+\int^t_0 Z_s\,dB_s\quad t\in[0,T]
\]
and
\[
Z^n\rightarrow Z,\quad\lambda\otimes P\mbox{-a.e.},
\quad\int^{T}_0 |f(s,Y^n_s,Z^n_s)-f(s,Y_s,Z_s)|\,ds\rightarrow
0\quad in\,\,\mbox{probability}\quad P.
\]
Moreover, there exists a stationary sequence
$\{\tau_k\}\subset\Gamma$ such that for every $p\in(0,2)$,
\begin{equation}
\label{eq2.2} E\int^{\tau_k}_0|Z^n_s-Z_s|^p\,ds\rightarrow 0.
\end{equation}
If $|\Delta^-C_t|+|\Delta^-K_t|=0$, $t\in[0,T]$, then
\mbox{\rm(\ref{eq2.2})} also holds for $p=2$.
\end{theorem}

\begin{remark}
Since the proof of the above theorem follows directly  from Lemma
\ref{kl4.11},  it suffices to assume in (e) that there exists a
stationary sequence $\{\tau_k\}$ such that $\sup_{n\ge
1}E\int^{\tau_k}_0|g_n(s)|^2\,ds<\infty$ (see Remark
\ref{uw1.kl1}).
\end{remark}

\nsubsection{Reflected BSDEs }

 In what follows we assume that the barrier $L$  is an
$\mathbb{F}$-adapted optional process and that $\xi\geq L_T$.

\begin{definition}
\label{def1} We say that  a triple $(Y, Z,K)$   of
$\mathbb{F}$-progressively measurable  processes is a solution of
the reflected backward stochastic differential equation with
right-hand side $f+dV$, terminal value $\xi$ and lower barrier $L$
(\textnormal{RBSDE($\xi,f+dV,L$)}) if
\begin{enumerate}[(a)]
\item $(Y,Z)\in\mathcal{S}^p\otimes \mathcal{H}$ for some $p>1$
or $Y$ is of class {\rm (D)} and $Z\in \mathcal{H}^q$ for $q\in
(0,1)$,
\item $K\in \mathcal{V}^+$,  $L_t\le Y_t$, $t\in[0,T]$, and
\[
\int^T_0(Y_{s-}-\limsup_{u\uparrow s}L_u)\,dK^*_s
+\sum_{s<T}(Y_s-L_s)\Delta^+K_s=0,
\]
\item $\int^T_0|f(s,Y_s,Z_s)|\,ds<\infty$,
\item $Y_t=\xi+\int^T_t f(s,Y_s,Z_s)\,ds
+\int_t^T\,dK_s+\int_t^T\,dV_s-\int^T_t Z_s\,dB_s,\quad t\in
[0,T]$.
\end{enumerate}
\end{definition}

\begin{remark}\label{rem1}
Assume that  $(Y,Z,K)$   is a solution of RBSDE($\xi,f+dV,L$). Let
$a\in\R$, and let
\[
\tilde
\xi=e^{aT}\xi,\quad \tilde L_t=e^{at}L_t,\quad \tilde V_t
=\int_0^te^{as}dV^*_s+\sum_{s<t}e^{as}\Delta V^+_s,
\]
\[
\quad \tilde f(t,y,z)=e^{at}f(t,e^{-at}y,e^{-at}z)-ay
\]
and
\[
\tilde Y_t=e^{at}Y_t\quad, \tilde Z_t=e^{at}Z_t \quad \tilde
K_t=\int_0^te^{as}dK^\star_s +\sum_{s<t}e^{as}\Delta K^+_s.
\]
Then $(\tilde Y,\tilde Z, \tilde K)$ solves  RBSDE$(\tilde
\xi,\tilde f+d\tilde V,\tilde L)$. Therefore choosing $a$
appropriately we may assume that (H2) is satisfied with arbitrary
but fixed $\mu\in\R$.
\end{remark}

Let $\{\sigma^i_k\}$ be a finite sequence of stopping  times and
let $(Y^i,Z^i,A^i)$ be a solution of the following BSDE
\begin{align}
\label{eq1} Y^i_t&=\xi^i+\int^T_t
f^i(s,Y^i_s,Z^i_s)\,ds+\int^T_t\,dV^i_s
-\int^T_t Z^i_s\,dB_s\nonumber\\
&\quad+\sum_{t\le\sigma^i_k<T}(Y^i_{\sigma^i_k+}
+\Delta^+V^i_{\sigma^i_k}-L_{\sigma^i_k})^-,\quad t\in[0,T],\quad
i=1,2.
\end{align}

\begin{proposition}
\label{prop2.2} Assume that  $f^1$ satisfies  \mbox{\rm(H1)} and
\mbox{\rm(H2)}, $\xi^1\le\xi^2$, $f^1(t,Y^2_t,Z^2_t)\le
f^2(t,Y^2_t,Z^2_t)$, $dV^1_t\le dV^2_t$, $U^1_t\le U^2_t$,
$t\in[0,T]$, and
$\bigcup_k[[\sigma^1_k]]\subset\bigcup_k[[\sigma^2_k]]$. If
$(Y^1-Y^2)^+\in\mathcal{S}^p$ for some $p>1$, then $Y^1_t\le
Y^2_t$, $t\in[0,T]$.
\end{proposition}
\begin{proof}
Let $q>1$ be such that $(Y^1-Y^2)^+\in\mathcal{S}^q$ and
$p\in(1,q)$.  Without loss of generality we may assume that
$\mu=-\frac{4\lambda}{p-1}$. By (H1) and (H2),
\begin{align*}
&((Y^1_s-Y^2_s)^+)^{p-1}(f^1(s,Y^1_s,Z^1_s)-f^2(s,Y^2_s,Z^2_s))\\
&\quad\le((Y^1_s-Y^2_s)^+)^{p-1}(f^1(s,Y^1_s,Z^1_s)-f^1(s,Y^2_s,Z^2_s))\\
&\quad\le-\frac{4\lambda}{p-1}((Y^1_s-Y^2_s)^+)^p
+\lambda((Y^1_s-Y^2_s)^+)^{p-1}|Z^1_s-Z^2_s|
\end{align*}
for $s\in[0,T]$. By Corollary \ref{cor4}, for
$\tau,\sigma\in\Gamma$ such that $\tau\le\sigma$ we have
\begin{align*}
&((Y^1_{\tau}-Y^2_{\tau})^+)^p
+\frac{p(p-1)}{2}\int^{\sigma}_{\tau}((Y^1_s-Y^2_s)^+)^{p-2}
\mathbf{1}_{\{Y^1_s>Y^2_s\}}|Z^1_s-Z^2_s|^2\,ds\\
&\quad\le((Y^1_{\sigma}-Y^2_{\sigma})^+)^p+p\int^{\sigma}_{\tau}
((Y^1_s-Y^2_s)^+)^{p-1}(f^1(s,Y^1_s,Z^1_s)-f^2(s,Y^2_s,Z^2_s))\,ds\\
&\quad\quad+p\int^{\sigma}_{\tau}((Y^1_{s-}-Y^2_{s-})^+)^{p-1}
\,d(V^1_s-V^2_s)^*+p\sum_{\tau\le s<\sigma}((Y^1_s-Y^2_s)^+)^{p-1}
\Delta^+(V^1_s-V^2_s)\\
&\quad\quad+p\sum_{\tau\le\sigma^1_k<\sigma}
((Y^1_{\sigma^1_k}-Y^2_{\sigma^1_k})^+)^{p-1}
(Y^1_{\sigma^1_k+}+\Delta^+V^1_{\sigma^1_k}-L_{\sigma^1_k})^-\\
&\quad\quad-p\sum_{\tau\le\sigma^2_k<\sigma}
((Y^1_{\sigma^2_k}-Y^2_{\sigma^2_k})^+)^{p-1}
(Y^2_{\sigma^2_k+}+\Delta^+V^2_{\sigma^2_k}-L_{\sigma^2_k})^-\\
&\quad\quad-p\int^{\sigma}_{\tau}((Y^1_s-Y^2_s)^+)^{p-1}(Z^1_s-Z^2_s)\,dB_s.
\end{align*}
By the above and the assumptions,
\begin{align}
\label{eq2.2.0}
&((Y^1_{\tau}-Y^2_{\tau})^+)^p+\frac{p(p-1)}{2}
\int^{\sigma}_{\tau}((Y^1_s-Y^2_s)^+)^{p-2}
\mathbf{1}_{\{Y^1_s>Y^2_s\}}|Z^1_s-Z^2_s|^2\,ds\nonumber\\
&\quad\le((Y^1_{\sigma}-Y^2_{\sigma})^+)^p
-\frac{4\lambda}{p-1}\int^{\sigma}_{\tau}((Y^1_s-Y^2_s)^+)^p\,ds
+\lambda\int^{\sigma}_{\tau}
((Y^1_s-Y^2_s)^+)^{p-1}|Z^1_s-Z^2_s|\,ds\nonumber\\
&\quad\quad+p\sum_{\tau\le\sigma^1_k<\sigma}
((Y^1_{\sigma^1_k}-Y^2_{\sigma^1_k})^+)^{p-1}
(Y^1_{\sigma^1_k+}
+\Delta^+V^1_{\sigma^1_k}-L_{\sigma^1_k})^-\nonumber\\
&\quad\quad-p\sum_{\tau\le\sigma^2_k<\sigma}
((Y^1_{\sigma^2_k}-Y^2_{\sigma^2_k})^+)^{p-1}
(Y^2_{\sigma^2_k+}+\Delta^+V^2_{\sigma^2_k}-L_{\sigma^2_k})^-
\nonumber\\
&\quad\quad-p\int^{\sigma}_{\tau}((Y^1_s-Y^2_s)^+)^{p-1}
(Z^1_s-Z^2_s)\,dB_s.
\end{align}
Since $\bigcup_k[[\sigma^1_k]]\subset\bigcup_k[[\sigma^2_k]]$,
\begin{align*}
&\sum_{\tau\le\sigma^1_k<\sigma}
((Y^1_{\sigma^1_k}-Y^2_{\sigma^1_k})^+)^{p-1}(Y^1_{\sigma^1_k+}
+\Delta^+V^1_{\sigma^1_k}-L_{\sigma^1_k})^-\\
&\quad\quad-\sum_{\tau\le\sigma^2_k<\sigma}
((Y^1_{\sigma^2_k}-Y^2_{\sigma^2_k})^+)^{p-1}(Y^2_{\sigma^2_k+}
+\Delta^+V^2_{\sigma^2_k}-L_{\sigma^2_k})^-\\
&\quad\le \sum_{\tau\le\sigma^1_k<\sigma}
((Y^1_{\sigma^1_k}-Y^2_{\sigma^1_k})^+)^{p-1}\{(Y^1_{\sigma^1_k+}
+\Delta^+V^1_{\sigma^1_k}-L_{\sigma^1_k})^--(Y^2_{\sigma^1_k+}
+\Delta^+V^2_{\sigma^1_k}-L_{\sigma^1_k})^-\}=:I.
\end{align*}
We shall show that $I\le0$.  Under the assumption that
$Y^1_{\sigma^1_k}\le Y^2_{\sigma^1_k}$ the this  inequality is
obvious. Assume now that $Y^1_{\sigma^1_k}>Y^2_{\sigma^1_k}$. By
(\ref{3.2.1}),
\begin{equation}\label{eq2.2.1}
Y^i_{\sigma^i_k}=(Y^i_{\sigma^i_k+}+\Delta^+V^i_{\sigma^i_k})
\vee L_{\sigma^i_k},\quad i=1,2.
\end{equation}
We have $Y^1_{\sigma^1_k}>Y^2_{\sigma^1_k}\ge L_{\sigma^1_k}$. By
this and (\ref{eq2.2.1}),
$Y^1_{\sigma^1_k+}+\Delta^+V^1_{\sigma^1_k}\ge L_{\sigma^1_k}$.
Hence
$(Y^1_{\sigma^1_k+}+\Delta^+V^1_{\sigma^1_k}-L_{\sigma^1_k})^-=0$,
which implies that
\begin{equation}
\label{eq3} I=-\sum_{\tau\le\sigma^1_k<T}
((Y^1_{\sigma^1_k}-Y^2_{\sigma^1_k})^+)^{p-1}(Y^2_{\sigma^1_k+}
+\Delta^+V^2_{\sigma^1_k}-L_{\sigma^1_k})^-\le0.
\end{equation}
Note that
\begin{align*}
&p\lambda((Y^1_s-Y^2_s)^+)^{p-1}|Z^1_s-Z^2_s|\\
&\quad=p((Y^1_s-Y^2_s)^+)^{p-2}\mathbf{1}_{\{Y^1_s>Y^2_s\}}
(\lambda(Y^1_s-Y^2_s)^+|Z^1_s-Z^2_s|)\\
&\quad\le p((Y^1_s-Y^2_s)^+)^{p-2}\mathbf{1}_{\{Y^1_s>Y^2_s\}}
\big(\frac{4\lambda^2}{p-1}((Y^1_s-Y^2_s)^+)^2+\frac{p-1}{4}
|Z^1_s-Z^2_s|^2\big)\\
&\quad=\frac{4p\lambda^2}{p-1}((Y^1_s-Y^2_s)^+)^p
+\frac{p(p-1)}{4}((Y^1_s-Y^2_s)^+)^{p-2}
\mathbf{1}_{\{Y^1_s>Y^2_s\}}|Z^1_s-Z^2_s|^2
\end{align*}
for $s\in[0,T]$. From this and (\ref{eq2.2.0}), (\ref{eq3}) it
follows that
\begin{align}
\label{eq2}
&((Y^1_{\tau}-Y^2_{\tau})^+)^p+\frac{p(p-1)}{4}
\int^\sigma_{\tau}((Y^1_s-Y^2_s)^+)^{p-2}
\mathbf{1}_{\{Y^1_s>Y^2_s\}}|Z^1_s-Z^2_s|^2\,ds\nonumber\\
&\quad\le((Y^1_{\sigma}-Y^2_{\sigma})^+)^p-p\int^\sigma_{\tau}
((Y^1_s-Y^2_s)^+)^{p-1}(Z^1_s-Z^2_s)\,dB_s\nonumber\\&
\quad=((Y^1_{\sigma}-Y^2_{\sigma})^+)^p+M_\sigma-M_{\tau}.
\end{align}
Let $\{\sigma_k\}\in\Gamma$ be a fundamental sequence  for the
local martingale $M$. Changing $\sigma_k$ with $\sigma$ in the
above inequality, taking expected value and passing to the limit
with $k\rightarrow\infty$ we get $E(Y^1_{\tau}-Y^2_{\tau})=0$. By
The Section Theorem, $(Y^1_t-Y^2_t)^+=0$, $t\in[0,T]$.
\end{proof}

\begin{remark}
\label{uw2.1} Observe that if $f,\, f'$ do not depend on $z$ then
it  is enough to assume that $(Y-Y')^+$ is of class
\textnormal{(D)}.
\end{remark}

\begin{remark}
\label{uw2.0} Let $f^1$, $f^2$, $\xi^1$, $\xi^2$, $dV^1$, $dV^2$,
$\bigcup_k[[\sigma^1_k]]$, $\bigcup_k[[\sigma^2_k]]$ be satisfying
the same assumptions as in Proposition \ref{prop2.2}. Moreover,
assume that $f^1$ satisfies (Z) and $Z^1,Z^2\in
L^q((0,T)\otimes\Omega)$ for some $q\in(\alpha,1]$. Then
$(Y^1-Y^2)^+\in\mathcal{S}^p$ for some $p>1$.
\end{remark}
\begin{proof}
By Corollary (\ref{cor4}), assumptions on the data and
(\ref{eq3}),
\begin{align*}
(Y^1_t-Y^2_t)^+&\le(\xi^1-\xi^2)^+
+\int^T_t\mathbf{1}_{\{Y^1_s>Y^2_s\}}(f^1(s,Y^1_s,Z^1_s)
-f^2(s,Y^2_s,Z^2_s))\,ds\\
&\quad+\int^T_t\mathbf{1}_{\{Y^1_{s-}>Y^2_{s-}\}}\,d(V^1_s-V^2_s)^*
+\sum_{t\le s<T}\mathbf{1}_{\{Y^1_s>Y^2_s\}}\Delta^+(V^1_s-V^2_s)\\
&\quad-\int^T_t\mathbf{1}_{\{Y^1_s>Y^2_s\}}(Z^1_s-Z^2_s)\,dB_s\\
&\quad+\sum_{t\le\sigma^1_k<T}\mathbf{1}_{\{Y^1_s>Y^2_s\}}
(Y^1_{\sigma^1_k+}+\Delta^+V^1_{\sigma^1_k}-L_{\sigma^1_k})^-\\
&\quad-\sum_{t\le\sigma^2_k<T}\mathbf{1}_{\{Y^1_s>Y^2_s\}}
(Y^2_{\sigma^2_k+}+\Delta^+V^2_{\sigma^2_k}-L_{\sigma^2_k})^-\\
&\le\int^T_t\mathbf{1}_{\{Y^1_s>Y^2_s\}}(f^1(s,Y^1_s,Z^1_s)
-f^1(s,Y^2_s,Z^2_s))\,ds\\
&\quad-\int^T_t\mathbf{1}_{\{Y^1_s>Y^2_s\}}(Z^1_s-Z^2_s)\,dB_s.
\end{align*}
Note that by (Z),
\begin{align*}
&|f^1(s,Y^1_s,Z^1_s)-f^1(s,Y^2_s,Z^2_s)|
\le|f^1(s,Y^1_s,Z^1_s)-f^1(s,Y^2_s,0)|\\
&\quad+|f^1(s,Y^2_s,0)-f^1(s,Y^2_s,Z^2_s)|
\le2\gamma(g_s+|Y^1_s|+|Y^2_s|+|Z^1_s|+|Z^2_s|)^{\alpha}
\end{align*}
for $s\in[0,T]$. Hence
\begin{align*}
(Y^1_t-Y^2_t)^+\le2\gamma
E\Big(\int^T_0(g_s+|Y^1_s|+|Y^2_s|+|Z^1_s|
+|Z^2_s|)^{\alpha}\,ds|\mathcal{F}_t\Big).
\end{align*}
Let $p>1$ be such that $\alpha\cdot p=q$. By Doob's inequality,
\begin{align*}
E\sup_{t\le T}((Y^1_t-Y^2_t)^+)^p
\le C_pE\Big(\int^T_0(g_s+|Y^1_s|+|Y^2_s|+|Z^1_s|+|Z^2_s|)^q\,ds\Big).
\end{align*}
Hence $(Y^1-Y^2)^+\in\mathcal{S}^p$.
\end{proof}

\begin{lemma}
\label{lm2.1} Let $x:[0,T]\rightarrow \BR$ be nonnegative, and
measurable and $y:[0,T]\rightarrow \BR$ be nondecreasing and
continuous. If for every $t\in(0,T]$ such that $x(t)>0$ there
exists $\varepsilon_t>0$ such that $\int_{t-\varepsilon_t}^t
x(s)\,dy(s)=0$,  then $\int_0^T x(s)\,dy(s)=0$.
\end{lemma}
\begin{proof}
Suppose that $\int_0^T x(s)\,dy(s)>0$. Set  $F(t)=\int_0^t
x(s)\,dy(s),\, t\in [0,T]$. It is well known that the function
\begin{equation}
\label{eq11.1}
f(t)\equiv\liminf_{\varepsilon\searrow 0}
\frac{F(t)-F(t-\varepsilon)}{y(t)-y(t-\varepsilon)}
\end{equation}
is Borel measurable and $f=x,\,dy$-a.e. Let
\[
A=\{t\in (0,T]: x(t)>0\},\quad B=\{t\in (0,T]:f(t)=x(t)\}.
\]
By the assumption, $dy(A\cap B)>0$. Let $t\in A\cap B$. Then
$x(t)>0$ and  by (\ref{eq11.1}), $\int_{t-\varepsilon}^t
x(s)\,dy(s)>0$ for every $\varepsilon >0$, which contradicts the
assumption.
\end{proof}

\begin{proposition}
\label{prop2.3} Assume that $Y$ is the Snell envelope of the
optional process $L$ of class \textnormal{(D)}. Let $A$ be the
continuous part of the increasing process $K$ from  Mertens
decomposition of $Y$. Then
\begin{equation}
\label{eq2.3.0}
\int_0^T(Y_r-\underline{L}_r)\,dA_r=0,
\end{equation}
where $\underline{L}_t=\limsup_{s\nearrow t}L_s$.
\end{proposition}
\begin{proof}
We may assume that $L$ is nonnegative, otherwise since $L$ is of
class (D) there exists uniformly integrable martingale $M$ such
that $L+M$ is nonnegative. We consider then $\tilde{L}=L+M$. Its
Snell envelope is equal to $\tilde{Y}=Y+M$, and obviously the
finite variation part of Mertens decomposition of $\tilde{Y}$ is
equal to the finite variation part of  Mertens decomposition of
$Y$, so its continuous parts are also equal. Therefore if we prove
that the assertion of the proposition holds for $\tilde{L}$ then
we would have
\[
\int_0^T(Y_r-\underline{L}_r)\,dA_r=\int_0^T(\tilde{Y}_r
-\underline{\tilde{L}}_r)\,dA_r=0.
\]
By \cite[Proposition 2.34, p.~131]{EK},  for any $t\in [0,T)$ and
$\lambda>0$,
\begin{equation}
\label{eq1.2}
\int_t^{D^\lambda_t}(Y_r-\underline{L}_r)\,dA_r=0,\quad P\mbox{-a.s.},
\end{equation}
where $D^\lambda_t=\inf\{r\ge t,\, \lambda Y_r\le L_r\}\wedge T$.
Let $\Omega_{t,\lambda}$ be the set of those $\omega\in\Omega$ for
which the above equality holds. Set
\[
\Omega_0=\bigcap_{t\in [0,T)\cap\mathbb{Q},\,\lambda\in
\mathbb{Q}^+} \Omega_{t,\lambda}.
\]
It is obvious that $P(\Omega_0)=1$.  We will show that for every
$\omega\in\Omega_0$ the following property holds:
\[
\forall{t\in (0,T]: Y_t>\underline{L}_t}\quad
\exists{\varepsilon_t>0}\quad
\int_{t-\varepsilon_t}^t(Y_r-\underline{L}_r)\,dA_r=0,
\]
which when combined with Lemma \ref{lm2.1} implies
(\ref{eq2.3.0}). Suppose that there exists $t\in (0,T]$ such that
\begin{equation}
\label{eq1.3}
Y_t>\underline{L}_t,\quad
\int_{t-\varepsilon}^t(Y_r-\underline{L}_r)\,dA_r>0,\quad \varepsilon>0.
\end{equation}
By the definition, $\underline{L}_t=\lim_{\delta\searrow
0}\sup_{t-\delta\le s<t}L_s$. Therefore there exist $\varepsilon,
\delta_1>0$ such that
\[
Y_t\ge\sup_{t-\delta_1\le s<t}L_s+2\varepsilon.
\]
Since $Y$  has only negative jumps,  there exists  $\delta_2>0$ such that
\[
Y_r\ge\sup_{t-\delta_1\le s<t}L_s+\varepsilon,\quad r\in [t-\delta_2,t].
\]
Let $\delta=\max\{\delta_1,\delta_2\}$ and $t_\delta=t-\delta$.
Recall that $D^\lambda_{t_\delta}=\inf\{r\ge t_\delta,\, \lambda
Y_r\le L_r\}\wedge T. $ Hence $D^\lambda_{t_\delta}\ge t$ for
$\lambda=(\sup_{r\in [t_\delta,t]} L_r+\varepsilon/2)/\inf_{r\in
[t_\delta,t]}Y_r$. It is clear that we can choose $\varepsilon,
\delta$ so that $\lambda, t_\delta$ are rational. Therefore  from
(\ref{eq1.2}) it follows that
\[
\int_{t_\delta}^t(Y_r-\underline{L}_r)\,dA_r
\le \int_{t_\delta}^{D^\lambda_{t_\delta}}(Y_r-\underline{L}_r)\,dA_r=0,
\]
which contradicts (\ref{eq1.3}).
\end{proof}

\begin{corollary}\label{wn2.1.0}
Let  $Y$ be the Snell envelope of an optional process $L$ of class
\textnormal{(D)}, and let  Let $K$ be an increasing process from
Mertens decomposition of  $Y$. Then
\[
\int^T_0(Y_r-\underline{L}_r)\,dK^*_r
=\sum_{t<T}(Y_t-\underline{L}_t)\Delta^+K_t=0.
\]
\end{corollary}
\begin{proof}
By \cite[Proposition 2.34, p.~131]{EK} we have
\[
\sum_{t\le T}(Y_{t-}-\underline{L}_{t})\Delta^-K_t+
\sum_{t<T}(Y_t-L_t)\Delta^+K_t=0.
\]
Therefore the desired result follows from  Proposition
\ref{prop2.3}.
\end{proof}

For optional processes $Y,Z$ we set
\[
f_{Y,Z}(t)=f(t,Y_t,Z_t),\, t\in [0,T].
\]

\begin{proposition}
\label{prop2.5} Let a triple $(Y,Z,K)$ be a solution of
\textnormal{RBSDE}$(\xi,f+dV,L)$ such that
$\int_0^T|f_{Y,Z}(s)|\,ds\in L^1$. Assume that $L^+$ is of class
\textnormal{(D)}, $\xi\in L^1$, $V\in\mathcal{V}^1$. Then for
$t\in[0,T]$,
\[
Y_t=\mathrm{ess}\sup_{\tau\in\Gamma_t}
E\big(\int_t^\tau f(s,Y_s,Z_s)\,ds+\int_t^\tau\,dV_s
+L_{\tau}{\bf 1}_{\{\tau<T\}}+\xi{\bf 1}_{\{\tau=T\}}|{\cal F}_t\big),
\]
where $\Gamma_t$ is set of all stopping times taking values in $[t,T]$.
\end{proposition}
\begin{proof}
It follows from the definition of solution of RBSDE$(\xi,f+dV,L)$
and Corollary \ref{wn2.1.0}.
\end{proof}
For a given process $L$ of class (D) and integrable
$\FF_T$-measurable random  variable $\xi$ we denote by
$\mbox{Snell}_\xi(L)$ the smallest supermartingale $Z$ such that
$Z_t\ge L_t,\, t\in [0,T)$ and $Z_T=\xi$. It is easy to see that
$\mbox{Snell}_\xi(L)=\mbox{Snell}(L^\xi)$, where
$L^\xi_t=\mathbf{1}_{\{t<T\}}L_t+\mathbf{1}_{\{t=T\}}\xi$. From
Proposition \ref{prop2.5} it follows that $\mbox{Snell}_\xi(L)$ is
the first component of the solution of RBSDE$(\xi,0,L)$.

\begin{proposition}
\label{prop2.4} Assume that there is a progressively measurable
process $g$ such that $E\int_0^T|g(r)|\,dr<\infty$ and
$f(r,y,z)\ge g(r)$ for a.e. $r\in [0,T]$ and all $y\in\BR$,
$z\in\BR^d$.  Let
\[
\hat{L}=\mbox{\rm Snell}_\xi(L+X)-X,
\]
where $(X,\tilde{Z})$ is a solution of
\textnormal{BSDE}$(0,-g-dV)$. If a triple $(Y,Z,K)$ is a solution
of \textnormal{RBSDE}$(\xi,f+dV,\hat{L})$ with the property that
$\int^T_0|f_{Y,Z}(s)|\,ds\in L^1$, then $(Y,Z,K)$ is a solution of
\textnormal{RBSDE}$(\xi,f+dV,L)$.
\end{proposition}
\begin{proof}
Let $(\bar{Y},\bar{Z},\bar{K})$ be a solution of
RBSDE$(\xi,f_{Y,Z}+dV,L)$. Then $\bar{Y}+X$ is a supermartingale
such that $\bar{Y}_T+X_T=\xi$ and $\bar{Y}_t+X_t\ge L_t+X_t,\,
t\in [0,T)$. Thus $\bar{Y}_t+X_t\ge \mbox{Snell}_\xi(L+X)_t,\,
t\in [0,T]$, and hence $\bar{Y}_t\ge
\mbox{Snell}_\xi(L+X)_t-X_t=\hat{L}_t,\, t\in [0,T]$. Moreover,
\[
\int_0^T(\bar{Y}_{t-}-\underline{\hat{L}}_t)\,d\bar{K}^*_t
+\sum_{t<T}(\bar{Y}_t-\hat{L}_t)\Delta^+\bar{K}_t\le
\int_0^T(\bar{Y}_{t-}-\underline{L}_t)\,d\bar{K}^*_t
+\sum_{t<T}(\bar{Y}_t-L_t)\Delta^+\bar{K}_t=0.
\]
Therefore  $(\bar{Y},\bar{Z},\bar{K})$ is also a solution of
RBSDE$(\xi,f_{Y,Z}+dV,\hat{L})$. By uniqueness (see Remark
\ref{uw2.0}), $(\bar{Y},\bar{Z},\bar{K})=(Y,Z,K)$. Therefore
$(Y,Z,K)$ is a solution of RBSDE$(\xi,f_{Y,Z}+dV,L)$ or,
equivalently, $(Y,Z,K)$ is a solution of RBSDE$(\xi,f+dV,L)$.
\end{proof}

\begin{lemma}
\label{lm2.2} Let $L$ be a regulated process such that
$\Delta^{-}(L+V)_t\le 0$ for $t\in (0,T]$,  and let
$(\bar{Y},\bar{Z},\bar{K})$ be a solution of
\textnormal{RBSDE}$(\xi,f+dV_+,L_+)$ such that
$\int^T_0|f_{\bar{Y},\bar{Z}}(s)|\,ds\in L^1$, where $L_+$ denotes
a c\`adl\`ag process defined by $(L_+)_t=L_{t+}$. Then
\[
(Y_+,Z,K_+)=(\bar{Y},\bar{Z},\bar{K}),
\]
where $(Y,Z,K)$ is a solution of \textnormal{RBSDE}$(\xi,f_{\bar{Y},\bar{Z}}+dV,L)$.
\end{lemma}
\begin{proof}
We will show that $(Y_+,Z,K_+)$ is a solution of
RBSDE$(\xi,f_{\bar{Y},\bar{Z}}+dV_+,L_+)$. Since $Y\ge L$, then of
course $Y_+\ge L_+$. Therefore it suffices to show that
\[
LS:=\int_0^T ((Y_{t+})_{-}-(L_{t+})_-)\, dK_{t+}=0.
\]
We have
\[
LS=\int_0^T(Y_{t-}-L_{t-})\,dK_{t+}=\int_0^T(Y_{t}-L_t)\,dK^c_t
+\sum_{0<t\le T}(Y_{t-}-L_{t-})\Delta K_{t+}.
\]
The first term on the right-hand side is equal to zero since
$(Y,Z,K)$ is a solution of RBSDE$(\xi,f_{\bar{Y},\bar{Z}}+dV,L)$.
As for the second term, we will consider two cases. First suppose
that $\Delta K_{t+}>0$ and $\Delta^-{K_t}>0$. Then $Y_{t-}=L_{t-}$
by the definition of a solution of
RBSDE$(\xi,f_{\bar{Y},\bar{Z}}+dV, L)$. Now suppose that $\Delta
K_{t+}>0$ and $\Delta^-{K_t}=0$. Then $\Delta^+K_t>0$.
Consequently, $Y_t=L_t$ by the definition of a solution of
RBSDE$(\xi,f_{\bar{Y},\bar{Z}}+dV,L)$. By the assumptions,
$L_{t-}+V_{t-}\ge L_t+V_t$. Hence
\[
Y_{t-}+V_{t-}\ge L_{t-}+V_{t-}\ge L_t+V_t=Y_t+V_t.
\]
But $Y_{t-}+V_{t-}=Y_t+V_t$, since $\Delta^-K_t=0$. Therefore
$Y_{t-}=L_{t-}$. Thus, in both cases, $Y_{t-}=L_{t-}$. Hence
$\sum_{0<t\le T}(Y_{t-}-L_{t-})\Delta K_{t+}=0$, and the proof is
complete.
\end{proof}

\begin{corollary}
\label{cor2.1} Let $p\ge 1$. Assume that \textnormal{(H1)--(H5)}
are satisfied and there exists a progressively measurable process
$g$ such that $\int_0^T|g(s)|\,ds \in \mathcal{H}^p$ and
$f(r,y,z)\ge g(r)$ for a.e. $r\in [0,T]$. If $p>1$ and
$L^+\in\mathcal{S}^p$  or $p=1$, $L^+$ is of class
\textnormal{(D)} and \textnormal{(Z)} is satisfied, then there
exists a unique solution $(Y,Z,K)$ of
\textnormal{RBSDE}$(\xi,f+dV,L)$. Moreover, $Y\in\mathcal{S}^p$,
$Z\in\mathcal{H}^p$, $K\in\mathcal{S}^p$ if $p>1$, and if $p=1$,
then $Y$ is of class \textnormal{(D)}, $Y\in\mathcal{S}^q$,
$Z\in\mathcal{H}^q$ for $q\in(0,1)$, $K\in\mathcal{V}^+$.
\end{corollary}
\begin{proof}
Define $X,\hat L$ as in Proposition \ref{prop2.4}. By Theorem
\ref{tw1.kl5} and Theorem \ref{tw1.kl6} there exists a solution of
$(\bar{Y},\bar{Z},\bar{K})$ of RBSDE$(\xi,f+dV_+,\hat{L}_+)$. By
Lemma \ref{lm2.2},
\[
(\bar{Y},\bar{Z},\bar{K})=(Y_+,Z,K_+),
\]
where $(Y,Z,K)$ is a solution of
RBSDE$(\xi,f_{\bar{Y},\bar{Z}}+dV,\hat{L})$.  Hence $(Y,Z,K)$ is a
solution of RBSDE$(\xi,f+dV,\hat{L})$, and by Proposition
\ref{prop2.4}, it  is a solution of RBSDE$(\xi,f+dV,L)$.
Uniqueness follows from  Proposition \ref{prop2.2} and Remark
\ref{uw2.0}.
\end{proof}

\begin{corollary}
\label{cor2.2} Under the assumptions of Corollary \ref{cor2.1},
\[
Y^n_t\nearrow Y_{t+},\quad t\in [0,T),
\]
where $(Y,Z,K)$ is the solution of {\rm RBSDE}$(\xi,f+dV,L)$ and
$(Y^n,Z^n)$ is the solution of the \textnormal{BSDE}
\[
Y^n_t=\xi+\int_t^Tf(s,Y^n_s,Z^n_s)\,ds
+\int_t^Tn(Y^n_s-\hat{L}_s)^-\,ds+\int_t^T\,dV_s-\int_t^TZ^n_s\,dB_s,
\quad t\in [0,T]
\]
with $\hat L$ defined in Proposition \ref{prop2.4}.
\end{corollary}

\begin{lemma}
\label{lm2.3} Let $p\ge1$. Assume that {\rm (H1)--(H5)} are
satisfied if $p>1$, and  {\rm (H1)--(H5), (Z)} are satisfied if
$p=1$. Let $(Y^1,Z^1,K^1)$, $(Y^2,Z^2,K^2)$ be solutions of
\textnormal{RBSDE}$(\xi^1,f^1+dV^1,L)$ and
\textnormal{RBSDE}$(\xi^2,f^2+dV^2,L)$, respectively. Assume that
$\xi^1\le \xi^2,\, f^1\le f^2$, $dV^1\le dV^2$ and there exists a
progressively measurable process $g$ such that $\int^T_0
|g(s)|\,ds\in L^1$ and $f^1(r,y)\wedge f^2(r,y)\ge g(r)$ for a.e.
$r\in [0,T]$. Then $ Y^1_t\le Y^2_t$, $t\in [0,T] $, and $dK^1\ge
dK^2$.
\end{lemma}
\begin{proof}
By Remark \ref{uw2.1}, $Y^1\le Y^2$. By Lemma \ref{lm2.2},
Proposition \ref{prop2.4} and \cite{kl}, $dK^1_+\ge dK^2_+$. Hence
$dK^{1,c}\ge dK^{2,c}$. Moreover,
\[
\Delta^+K^1_t=(\hat{L}_{t}-Y^1_{t+}-\Delta^+V^1_t)^+
\ge (\hat{L}_t-Y^2_{t+}-\Delta^+V^2_t)=\Delta^+K^2_t,
\]
\[
\Delta^-K^1_t=(\hat{L}_{t-}-Y^1_{t}-\Delta^-V^1_t)^+
\ge (\hat{L}_{t-}-Y^2_{t}-\Delta^- V^2_t)^+=\Delta^-K^2_t.
\]
\end{proof}

\begin{lemma}\label{lem2}
Assume that $E\int^T_{0}|f_n(s)-f(s)|\,ds\rightarrow 0$,
$E|\xi^n-\xi|\rightarrow 0$, $\| L-L^n\|_D\rightarrow 0$. Let
\[
Y^n_t=\esssup_{\tau\ge t}E(\int_{t}^\tau f_n(s)\,ds
+L^n_{\tau}\mathbf{1}_{\{\tau<T\}}+\xi_n\mathbf{1}_{\{\tau=T\}}|\mathcal{F}_t).
\]
Then $\|Y^n-Y\|_D\rightarrow 0$, where
\[
Y_t=\esssup_{\tau\ge t}E(\int_{t}^\tau f(s)\,ds
+L_{\tau}\mathbf{1}_{\{\tau<T\}}+\xi\mathbf{1}_{\{\tau=T\}}|\mathcal{F}_t).
\]
\end{lemma}
\begin{proof}
For every $\sigma\in\Gamma$ we have
\begin{align*}
&E|Y_{\sigma}-Y^n_{\sigma}|\quad\\
&\quad \le E \esssup_{\tau\ge\sigma}E(|\int_{\sigma}^\tau
f(s)-f_n(s)\,ds
+(L_{\tau}-L^n_{\tau})\mathbf{1}_{\{\tau<T\}}+(\xi-\xi^n)
\mathbf{1}_{\{\tau=T\}}|\mathcal{F}_{\sigma})\\
&\quad=\sup_{\tau\ge\sigma}E \big(E(|\int_{\sigma}^\tau f(s)-f_n(s)\,ds
+(L_{\tau}-L^n_{\tau})\mathbf{1}_{\{\tau<T\}}
+(\xi-\xi^n)\mathbf{1}_{\{\tau=T\}}|\mathcal{F}_{\sigma})\big)\\
&\quad\le\sup_{\tau\in\Gamma}E|L_{\tau}-L^n_{\tau}|
+E\int^T_0|f(s)-f_n(s)|\,ds+E|\xi-\xi^n|,
\end{align*}
which converges to 0 as $n\rightarrow$ by the assumptions of the
lemma.
\end{proof}

\begin{theorem}\label{tw2.16}
Let $p\ge1$. Assume that \textnormal{(H1)--(H6)} are satisfied if
$p>1$, and if $p=1$  then \textnormal{(H1)-- (H5), (H6*), (Z)} are
satisfied. Then there exists a unique solution $(Y,Z,K)$ of {\rm
RBSDE($\xi$,$f+dV$,$L$)}. Moreover,  $Y\in\mathcal{S}^p$,
$Z\in\mathcal{H}^p$ and $K\in\mathcal{S}^p$ if $p>0$, and if
$p=1$, then $Y$ is of class \textnormal{(D)}, $Y\in\mathcal{S}^q$,
$Z\in\mathcal{H}^q$ for $q\in(0,1)$ and $K\in\mathcal{V}^+$.
\end{theorem}
\begin{proof}
Let $f_n(t,y,z)=f(t,y,z)\vee(-n)$. By Corollary \ref{cor2.1}, for
$n\ge 1$ there exists a solution $(Y^n,Z^n,K^n)$ of
RBSDE($\xi$,$f_n+dV$,$L$). By Lemma \ref{lm2.3}, $Y^n\ge Y^{n+1}$
and $dK^n\le dK^{n+1}$, $n\ge 1$. By this and Proposition
\ref{prop2.2},
\begin{equation}\label{tw2.16.1}
\bar{Y}\le Y^n\le Y^1,\quad n\ge 1,
\end{equation}
where $(\bar{Y},\bar{Z})$ is a solution of BSDE($\xi$,$f+dV$). By
the above  (H2) we have
\begin{equation}\label{tw2.16.2}
|f_n(s,Y^n_s,0)|\le|f(s,Y^1_s,0)|+|f(s,\bar{Y}_s,0)|.
\end{equation}
Let $\tau^1_k=\inf\{t\ge
0:\int^t_0|f(s,Y^1_s,0)|\,ds+\int^t_0|f(s,\bar{Y},0)|\,ds>k\}$,
and let $\{\tau^2_k\}\subset\Gamma$ be a stationary sequence of
stopping times such that
$Y^{1,\tau^2_k},\bar{Y}^{\tau^2_k},V^{\tau^2_k}\in\mathcal{S}^2$,
$\int^{\tau^2_k}_0|f(s,0,0)|\,ds\in L^2$. Write
$\tau_k=\tau^1_k\wedge\tau^2_k$, $k\in\mathbb{N}$. By \cite[Lemma
4.2]{kl} and the definition of $\{\tau_k\}$, for $q\le2$ we have
\begin{align}
\label{tw2.16.3} &E\Big(\int^{\tau_k}_0|Z^n_s|^2\,ds\Big)^{q/2}
+E\Big(\int^{\tau_k}_0\,dK^n_s\Big)^q\nonumber\\
& \le C\Big(E\sup_{0\le t\le\tau_k}|Y^1_t|^q +E\sup_{0\le
t\le\tau_k}|\bar{Y}_t|^q +E\Big(\int^{\tau_k}_0\,d|V|_s\Big)^q
+E\Big(\int^{\tau_k}_0 f_n^-(s,Y^n_s,0)\,ds\Big)^q\Big)\nonumber\\
&\le C\Big(E\sup_{0\le t\le\tau_k}|Y^1_t|^q+E\sup_{0\le
t\le\tau_k}|\bar{Y}_t|^q+(2k)^q+\Big(\int^{\tau_k}_0\,d|V|_s\Big)^q\Big).
\end{align}
Set $g_n(s)=f_n(s,Y^n_s,0)$,
$h_n(s)=f_n(s,Y^n_s,Z^n_s)-f_n(s,Y^n_s,0)$. From the above, the
the definition of $\{\tau_k\}$ and (\ref{tw2.16.2}) it follows
that $g_n,h_n$  satisfy the assumptions of Lemma \ref{kl4.11} (see
also Remark \ref{uw1.kl1}). Hence, for $q<2$,
\[
E\int^{\tau_k}_0|Z^n_s-Z_s|^q\,ds\rightarrow 0,
\]
and, by stationarity of $\{\tau_k\}$, $Z^n\rightarrow Z$ in
measure $\lambda\otimes P$ on $[0,T]\times\Omega$. By this and by
(\ref{tw2.16.2}) and (\ref{tw2.16.3}),
\begin{equation}
\label{eq3.14} Y_t=\xi+\int^T_t f(s,Y_s,Z_s)\,ds
+\int^T_t\,dK_s+\int^T_t\,dV_s-\int^T_t Z_s\,dB_s,\quad t\in[0,T],
\end{equation}
where $Y_t=\lim_{n\rightarrow\infty}Y^n_t$,
$K_t=\lim_{n\rightarrow\infty}K^n_t$. It is obvious that $Y$ is
regulated and $Y_t\ge L_t$, $t\in[0,T]$. We have to show the
minimality condition for $K$ and integrability of $Z$ and $K$. We
know that $\sum_{t<T}(Y^n_t-L_t)\Delta^+K^n_t=0$. Letting
$n\rightarrow\infty$ we obtain
\[
\sum_{t<T}(Y_t-L_t)\Delta^+K_t=0.
\]
Therefore to prove the minimality condition for $K$ it suffices to
show that
\begin{equation}
\label{eq3.13} \int^T_0(Y_{t-}-\underline{L}_{t})\,dK^*_t=0
\end{equation}
where $\underline{L}_{t}$ is defined as in Proposition
\ref{prop2.3}. Note that
\[
\int^T_0(Y^n_{t-}-\underline{L}_{t})\,dK^{n,*}_t
=\int^T_0(Y^n_t-\underline{L}_t)\,dK^{n,c}_t
 +\sum_{0< t\le T}(Y^n_{t-}-\underline{L}_{t})\Delta^-K^n_t.
\]
We know that $dK^n\rightarrow dK$ in the total variation norm and
that $0\le Y^n_t-\underline{L}_t\le Y^1_t-\underline{L}_t$.
Therefore applying the Lebesgue dominated convergence theorem we
get $\int^T_0(Y_t-\underline{L}_t)\,dK^c_t=0$. This gives
(\ref{eq3.13}) if $\Delta^{-}K=0$. If $\Delta^-K_t>0$ for some
$t\in(0,T]$, then there exists $N\in\mathbb{N}$ such that
$\Delta^-K^n_t>0$ for $n\ge N$. Hence
$Y^n_{t-}=\underline{L}_{t}$, $n\ge N$. By Proposition
\ref{prop2.2} and Remark \ref{uw2.0},  $Y_{t-}\le
Y^n_{t-}=\underline{L}_{t}$, and consequently,
$Y_{t-}=\underline{L}_{t}$. Hence
\[
\sum_{t\le T}(Y_{t-}-\underline{L}_t)\Delta^-K_t=0,
\]
so (\ref{eq3.13}) is satisfied. This proves the the minimality
condition for $K$. Note that by (H6) the process $X$ is of the
form
\[
X_t=X_0+\int^t_0\,dC_s+\int^t_0 H_s\,dB_s
\]
for some $C\in\mathcal{V}^p$, $H\in\mathcal{H}$. It can be
rewritten in the form
\[
X_t=\xi+\int^T_t
f(s,X_s,H_s)\,ds+\int^T_t\,dC'_s+\int^T_t\,dV_s-\int^T_t H_s\,dB_s
\]
for some $C'\in\mathcal{V}^p$. Let $(\bar{X},\bar{Z})$ be a
solution of  BSDE($\xi$,$f+dV^++dC^{',+}$). By Proposition
\ref{prop2.2}, $\bar{X}\ge X$, so $\bar{X}\ge L$. Note that the
triple $(\bar{X},\bar{H},C^{',+})$ is a solution of
RBSDE($\xi$,$f+dV^+$,$\bar{X}$). Since $\bar{X}\ge L$, then by
Proposition \ref{prop2.2} for $p>1$, $\bar{X}\ge Y$. For $p=1$ we
can not for now apply  Proposition \ref{prop2.2}  since we do not
know a priori that $Z\in\mathcal{H}^q$ for some $q>\alpha$ (see
Remark \ref{uw2.0}). Let $(\bar{X}^n,\bar{H}^n)$ be a solution of
BSDE($\xi$,$f_n+dV^++dC^{',+}$). By Proposition \ref{prop2.2},
$\bar{X}^n\ge\bar{X}\ge L$. Hence, by Proposition \ref{prop2.2}
again,
\begin{equation}\label{tw2.16.4}
\bar{X}^n\ge Y^n,\quad n\ge 1.
\end{equation}
In the same manner as in the proof of (\ref{eq3.14}) we show that
$\bar{X}^n_t\searrow\tilde{X}_t$, $t\in[0,T]$,
$\bar{H}^n\rightarrow\tilde{H}$ in measure $\lambda\otimes P$ on
$[0,T]\times\Omega$, and
\[
\tilde{X}_t=\xi+\int^T_t f(s,\tilde{X}_s,\tilde{H}_s)\,ds
+\int^T_t\,dC^{',+}_s+\int^T_t\,dV^+_s-\int^T_t \tilde{H}_s\,dB_s.
\]
Since $\bar{Y}\le\tilde{X}\le\bar{X}^1$, it follows that
$\tilde{X}\in\mathcal{S}^q$, $q\in(0,1)$. By \cite[Lemma
3.1]{bdh},  $\tilde{Z}\in\mathcal{H}^q$, $q\in(0,1)$. Therefore by
Proposition \ref{prop2.2} and Remark \ref{uw2.0},
$\tilde{X}=\bar{X}$. By this and (\ref{tw2.16.4}), $\bar{X}\ge Y$
for $p=1$. By \cite[Lemma 4.2, Proposition 4.3]{kl} we have
integrability of $Y$, $Z$ and $K$ for $p\ge 1$.
\end{proof}

\nsubsection{Penalization method for reflected BSDEs}
\label{sec4}

We assume that the  barrier $L$ has regulated trajectories.  We
consider approximation of the solution of RBSDE($\xi$,$f+dV$,$L$)
by a modified penalization method of the form
\begin{align}
\label{eq3.1}
Y^n_t&=\xi+\int^T_tf(s,Y^n_s,Z^n_s)\,ds +\int^T_t\,dV_s
-\int^T_t Z^n_s,\,dB_s\nonumber\\
&\quad+n\int_t^T(Y^n_s-L_s)^-ds+\sum_{t\leq
\sigma_{n,i}<T}(Y^n_{\sigma_{n,i}+}
+\Delta^+V_{\sigma_{n,i}}-L_{\sigma_{n,i}})^-,\quad
t\in[0,T]
\end{align}
with specially defined arrays of stopping times
$\{\{\sigma_{n,i}\}\}$ exhausting  right-side jumps of $L$ and
$V$. We define $\{\{\sigma_{n,i}\}\}$ inductively.  We first set
$\sigma_{1,0}=0$ and
\[
\sigma_{1,i}=\inf\{t>\sigma_{1,i-1}:\Delta^+
L_s<-1\,\,\mathrm{or}\,\,\Delta^+ V_s<-1\}\wedge T,\quad
i=1,\dots,k_1
\]
for some $k_1\in\N$. Next, for $n\in\N$ and given array
$\{\{\sigma_{n,i}\}\}$ we set $\sigma_{n+1,0}=0$ and
\[
\sigma_{n+1,i}=\inf\{t>\sigma_{n+1,i-1}:
\Delta^+ L_s<-1/(n+1)\,\,\mathrm{or}\,\,\Delta^+
V_s<-1/(n+1)\}\wedge T
\]
for $i=1,\dots,j_{n+1}$, where $j_{n+1}$ is
chosen so that $P(\sigma_{n+1,j_{n+1}}<T)\to 0$ as $n\to\infty$ and
\[
\sigma_{n+1,i} =\sigma_{n+1,j_{n+1}}\vee
\sigma_{n,i-j_{n+1}},\quad
i=j_{n+1}+1,\dots,k_{n+1},\,k_{n+1}=j_{n+1}+k_{n}.\] Since
$\Delta^+L_s<-1/n$ or $\Delta^+V_s<-1/n$ implies that
$\Delta^+L_s<-1/(n+1)$ or $\Delta^+V_s<-1/(n+1)$, it follows from
our construction that
\begin{equation}
\label{eq3.2}
\bigcup_i [[\sigma_{n,i}]]\subset \bigcup_i [[\sigma_{n+1,i}]]\quad n\in\N.
\end{equation}
Moreover,  on  each interval $(\sigma_{n,i-1},\sigma_{n,i}]$,
$i=1,\dots,{k_n+1}$, where $\sigma_{n,{k_n+1}}=T$, the pair
$(Y^n,Z^n)$ is a solution of the classical generalized BSDEs  of
the form
\begin{align}
\label{3.2.1}
Y^n_t&=L_{\sigma_{n,i}}\vee(Y^n_{\sigma_{n,i}+}+\Delta^+
V_{\sigma_{n,i}})
+\int^{\sigma_{n,i}}_tf(s,Y^n_s,Z^n_s)\,ds+\int^{\sigma_{n,i}}_t\,dV_s\nonumber\\
&\quad+n\int^{\sigma_{n,i}}_t
(Y^n_s-L_s)^-\,ds-\int^{\sigma_{n,i}}_tZ^n_s\,dB_s,\quad
t\in(\sigma_{n,i-1},\sigma_{n,i}]
\end{align}
and $Y^n_0=L_{0}\vee(Y^n_{0+}+\Delta^+ V_{0})$, $n\in\N$. Observe
that (\ref{eq3.1})  can written in the following shorter form
\begin{equation}
\label{eq3.3}
Y^n_t=\xi+\int_t^Tf(s,Y^n_s,Z^n_s)\,ds+\int^T_t\,dV_s
+\int_t^T\,dK^n_s-\int_t^TZ^n_s\,dB_s,
\end{equation}
where
\[
K^n_t=n\int_0^t (Y^n_s-L_s)^-\,ds +\sum_{t\leq
\sigma_{n,i}<T}(Y^n_{\sigma_{n,i}+}
+\Delta^+V_{\sigma_{n,i}}-L_{\sigma_{n,i}})^-
=:K^{n,*}_t+K^{n,d}_t.
\]
For similar approximation scheme see  \cite{KRS}. As compared with
the usual penalization method, the term $K^n$ includes the purely
jumping part $K^{n,d}$ consisting of right jumps. If the processes
$L, V$ are right-continuous then $K^{n}=K^{n,*}$, so (\ref{eq3.1})
(or, equivalently, (\ref{eq3.3})) reduces to the usual
penalization scheme. Note that if $Y$ is a limit of increasing
sequence $\{Y^n\}$ of c\`adl\`ag solutions of BSDEs, then by the
monotone convergence theorem for BSDEs (see, e.g., \cite{peng}),
$Y$ is also c\`adl\`ag. On the other hand, if $L$ is a regulated
process, then in general the solution $Y$ need not be c\`adl\`ag.
Therefore the usual penalization equations have to be modified by
adding  right jumps corrections.

\begin{theorem}
Let $(Y^n,Z^n)$, $n\in\mathbb{N}$, be a solution of
{\rm(\ref{eq3.1})}.
\begin{enumerate}
\item[\rm(i)]Assume that $p>1$ and
\textnormal{(H1)--(H6)} are satisfied. Then $Y^n_t\nearrow Y_t$,
$t\in[0,T]$, and for any $\gamma\in[1,2)$,
\begin{equation}
\label{eq4.6}
E\Big(\int^T_0|Z^n_s-Z_s|^{\gamma}\,ds\Big)^{p/2}\rightarrow 0,
\end{equation}
where $(Y,Z,K)$ is unique solution of
\textnormal{RBSDE($\xi$,$f+dV$,$L$)}. Moreover, if $\Delta^-K_t=0$
for $t\in(0,T]$, then \mbox{\rm(\ref{eq4.6})} hold true with
$\gamma=2$.

\item[\rm(ii)]Assume that  $p=1$ and \textnormal{(H1)--(H5), (H6*), (Z)}
are satisfied. Then $Y^n_t\nearrow Y_t$, $t\in[0,T]$, and for any
$\gamma\in[1,2)$ and $r\in(0,1)$,
\begin{equation}
\label{eq4.7}
E\Big(\int^T_0|Z^n_s-Z_s|^{\gamma}\,ds\Big)^{r/2}\rightarrow 0,
\end{equation}
where $(Y,Z,K)$ is a unique solution of
\textnormal{RBSDE($\xi$,$f+dV$,$L$)}. If $\Delta^-K_t=0$ for
$t\in(0,T]$, then \mbox{\rm(\ref{eq4.7})} hold true with
$\gamma=2$.
\end{enumerate}
\end{theorem}
\begin{proof}
Let $p\ge 1$. Without loss of generality we may assume that
$\mu=0$.  Let $(Y^n,Z^n)$, $n\in\mathbb{N}$ be a solution of
(\ref{eq3.1}). By Proposition \ref{prop2.2}, $Y^n_t\le Y^{n+1}_t$,
$t\in[0,T]$, $n\in\mathbb{N}$. The rest of the proof we divide
into 3 steps.

{\em Step 1.} We first show that for $n\in\mathbb{N}$  the triple
$(Y^n,Z^n,K^n)$ is a solution of RBSDE$(\xi,f+dV,L^n)$ with
$L^n=L-(Y^n-L)^-$. Note that $Y^n_t\ge L^n_t$, $t\in[0,T]$.
Indeed, if $Y^n_t\ge L_t$ then $Y^n_t\ge L^n_t$, and if $Y^n_t<
L_t$ then $Y^n_t\ge Y^n_t=L_t$. Moreover,
\begin{align*}
\int_0^T (Y^n_{s-}-L^n_{s-})\,dK^{n,*}_s
&=n\int^T_0 (Y^n_s-L^n_s)(Y^n_s-L_s)^-\,ds\\
&=n\int^T_0 (Y^n_s-L_s)^+(Y^n_s-L_s)^-\,ds=0
\end{align*}
and
\begin{align*}
\sum_{s<T}(Y^n_s-L^n_s)\Delta^+K^n_s
&=\sum_{\sigma_{n,i}<T}(Y^n_{\sigma_{n,i}}
-L^n_{\sigma_{n,i}})(Y^n_{\sigma_{n,i}+}+\Delta^+V_{\sigma_{n,i}}
-L_{\sigma_{n,i}})^-\\
&=\sum_{\sigma_{n,i}<T}(Y^n_{\sigma_{n,i}}
-L_{\sigma_{n,i}})^+(Y^n_{\sigma_{n,i}+}
+\Delta^+V_{\sigma_{n,i}}-L_{\sigma_{n,i}})^-=0.
\end{align*}
Indeed, suppose that
\begin{equation}
\label{eq4.5} \sum_{\sigma_{n,i}<T}(Y^n_{\sigma_{n,i}}
-L_{\sigma_{n,i}})^+(Y^n_{\sigma_{n,i}+}+\Delta^+V_{\sigma_{n,i}}
-L_{\sigma_{n,i}})^-\neq 0.
\end{equation}
Then there is $1\le i\le k_n$ such that $Y^n_{\sigma_{n,i}}
-L_{\sigma_{n,i}}>0$ and
$Y^n_{\sigma_{n,i}+}+\Delta^+V_{\sigma_{n,i}}-L_{\sigma_{n,i}}<0$.
By the last inequality and (\ref{eq3.3}),
$\Delta^+Y^n_{\sigma_{n,i}}=\Delta^+K^n_{\sigma_{n,i}}
-\Delta^+V_{\sigma_{n,i}}=-(Y^n_{\sigma_{n,i}+}
+\Delta^+V_{\sigma_{n,i}}-L_{\sigma_{n,i}})^--\Delta^+V_{\sigma_{n,i}}$.
Hence $Y^n_{\sigma_{n,i}}=L_{\sigma_{n,i}}$, which contradicts
(\ref{eq4.5}).

{\em Step 2.} \noindent We now show that $Y_t:=\sup_{n\ge
1}Y^n_t$, $t\in[0,T]$, is a regulated process satisfying condition
(d) of Definition \ref{def1} and that $(Y,Z,K)$ has the desired
integrability properties. To this end, we first  prove that if
$p>1$ then (\ref{eq4.6}) holds true, and if $p=1$, then there
exists a stationary sequence of stopping times $\{\tau_k\}$ such
that for any $\gamma\in[1,2)$ and $r\in(0,1)$,
\[
E\Big(\int^{\tau_k}_0|Z^n_s-Z_s|^{\gamma}\,ds\Big)^{r/2}\rightarrow
0.
\]
To show this we will use \cite[Lemma 4.2]{kl}. Let $p>1$. Then by
(H6) there exists a process
$X\in(\mathcal{M}_{loc}+\mathcal{V}^p)\cap\mathcal{S}^p$ such that
$X\ge L$ and $\int^T_0 f^-(s,X_s,0)\,ds\in L^p$. If $p=1$ then by
(H6*) there exists  $X$ of class (D) such that
$X\in\mathcal{M}_{loc}+\mathcal{V}^1$, $X\ge L$ and $\int^T_0
f^-(s,X_s,0)\,ds\in L^1$. Since the Brownian filtration has the
representation property, there exist processes
$H\in\mathcal{M}_{loc}$ and $C\in\mathcal{V}^p$ such that
\[
X_t=X_T-\int^T_t\,dC_s-\int^T_t H_s\,dB_s,\quad t\in[0,T],
\]
which can be rewritten in form
\[
X_t=\xi+\int^T_t
f(s,X_s,H_s)\,ds+\int^T_t\,dV_s+\int^T_t\,dK'_s-\int^T_t\,dA'_s
-\int^T_t H_s\,dB_s
\]
for some $A',K'\in\mathcal{V}^{+,p}$ with $p\ge 1$. Let
$(\bar{X}^n,\bar{H}^n)$ be a solution of the BSDE
\begin{align*}
\bar{X}^n_t&=\xi+\int^T_t f(s,\bar{X}^n_s,\bar{H}^n_s)\,ds+\int^T_t\,dV_s
+\int^T_t\,dK'_s-\int^T_t \bar{H}^n_s\,dB_s\\
&\quad+\sum_{t\le\sigma_{n,i}<T}(\bar{X}^n_{\sigma_{n,i}+}
+\Delta^+V_{\sigma_{n,i}}-L_{\sigma_{n,i}})^-,\quad t\in[0,T].
\end{align*}
By Proposition \ref{prop2.2} and Remark \ref{uw2.0}, $\bar{X}^n\ge
X\ge L$, so we may rewrite the above equation in the form
\begin{align*}
\bar{X}^n_t&=\xi+\int^T_t f(s,\bar{X}^n_s,\bar{H}^n_s)\,ds
+\int^T_t\,dV_s+\int^T_t\,dK'_s+n\int^T_t(\bar{X}^n_s-L_s)^-\,ds\\
&\quad+\sum_{t\le\sigma_{n,i}<T}(\bar{X}^n_{\sigma_{n,i}+}
+\Delta^+V_{\sigma_{n,i}}-L_{\sigma_{n,i}})^-
-\int^T_t \bar{H}^n_s\,dB_s,\quad t\in[0,T].
\end{align*}
By Proposition \ref{prop2.2} and Remark \ref{uw2.0}, $\bar{X}^n\ge
Y^n$. Also note that
\begin{align*}
&(\bar{X}^n_{\sigma_{n,i}+}+\Delta^+V_{\sigma_{n,i}}
-L_{\sigma_{n,i}})^-\le (X_{\sigma_{n,i}+}+\Delta^+V_{\sigma_{n,i}}
-L_{\sigma_{n,i}})^-\\
&\quad=(\Delta^+X_{\sigma_{n,i}}+\Delta^+V_{\sigma_{n,i}}
+X_{\sigma_{n,i}}-L_{\sigma_{n,i}})^-\le(\Delta^+X_{\sigma_{n,i}}
+\Delta^+V_{\sigma_{n,i}})\\
&\quad\quad\le\Delta^+|C|_{\sigma_{n,i}}+\Delta^+|V|_{\sigma_{n,i}}.
\end{align*}
Let $(\tilde{X},\tilde{H})$ be a solution of the BSDE
\begin{align*}
\tilde{X}_t&=\xi+\int^T_t f(s,\tilde{X}_s,\tilde{H}_s)\,ds
+\int^T_t\,dV_s+\int^T_t\,dK'_s+n\int^T_t(\tilde{X}_s-L_s)^-\,ds\\
&\quad+\int^T_t\,d|C|_s+\int^T_t\,d|V|_s-\int^T_t \tilde{H}_s\,dB_s,
\quad t\in[0,T].
\end{align*}
The pair $(\tilde{X},\tilde{H})$ does not depend on $n$, because
by Proposition \ref{prop2.2} and Remark \ref{uw2.0},
$\tilde{X}\ge\bar{X}^n$, so the term involving $n$ on the
right-hand side of the above equation  equals zero. By the last
inequality we also have $\tilde{X}\ge Y^n$. Thus all the
assumptions of \cite[Lemma 4.2]{kl} are satisfied.  Applying
\cite[Lemma 4.2]{kl} we get
\begin{align}
\label{Tw.3.1} &E(K^n_T)^p+E\Big(\int^T_0|Z^n_s|^2\,ds\Big)^{p/2}
\le CE\Big(\sup_{t\le
T}(|Y^1_t|^p+|\tilde{X}_t|^p)+\Big(\int^T_0\,d|V|_s\Big)^p
\nonumber\\
&\quad+\Big(\int^T_0|f^-(s,\tilde{X}_s,0)|\,ds\Big)^p
+\Big(\int^T_0\tilde{X}^+_s\,ds\Big)^p
+\Big(\int^T_0|f(s,0,0)|\,ds\Big)^p\Big)
\end{align}
if $p>1$, which means that $\{Z^n\}$ is bounded in
$\mathcal{H}^p$. If $p=1$ then by \cite[Lemma 4.2]{kl}, for any
$q\in(0,1)$ we have
\begin{align}
\label{Tw.3.2}
&E\Big(\int^T_0|Z^n_s|^2\,ds\Big)^{q/2}
\le CE\Big(\sup_{t\le T}(|Y^1_t|^q+|\tilde{X}_t|^q)
+\Big(\int^T_0|f(s,0,0)|\,ds\Big)^q\nonumber\\
&\quad+\Big(\int^T_0|f^-(s,\tilde{X}_s,0)|\,ds\Big)^q
+\Big(\int^T_0\tilde{X}^+_s\,ds\Big)^q
+\Big(\int^T_0\,d|V|_s\Big)^q\Big).
\end{align}
We next check that the assumption of Theorem \ref{kl4.12} are
satisfied. We know that $Y^n$ is of class (D),
$Z^n\in\mathcal{H}$, $K^n\in\mathcal{V}^+$ and $t\mapsto
f(t,Y^n_t,Z^n_t)\in L^1(0,T)$. Since $V$ is a finite variation
process and $A^n=-V$, we have $A^n\le A^{n+1}$ and
$E|A^n|_T<\infty$ for $n\in\mathbb{N}$, i.e. assumption (a) is
satisfied. Let $\tau,\sigma\in\mathcal{T}$ be  stopping times such
that $\sigma\le\tau$. By the Lebesgue dominated convergence
theorem,
$\lim_{n\rightarrow\infty}\int^{\tau}_{\sigma}(Y_s-Y^n_s)\,dV^*_s=0$
and $\lim_{n\rightarrow\infty}\sum_{\sigma\le
s<\tau}(Y_s-Y^n_s)\Delta^+V_s=0$. Since we know that
$\int^{\tau}_{\sigma}(Y_s-Y^n_s)\,dK^{n,*}_s+\sum_{\sigma\le
s<\tau}(Y_s-Y^n_s)\Delta^+K^n_s\ge 0$, it follows that
\[
\liminf_{n\rightarrow \infty}
\Big(\int^{\tau}_{\sigma}(Y_s-Y^n_s)\,d(K^n_s-A^n_s)^*+\sum_{\sigma\le
s<\tau}(Y_s-Y^n_s)\Delta^+(K^n_s-A^n_s)\Big)\ge 0,
\]
i.e. (b) is satisfied. It is easy to see that $\Delta^-K^n_t=0$
for $n\in\mathbb{N}$ and $t\in[0,T]$, so (c) is satisfied. Let
$\bar{y}=Y^1$ and $\underline{y}=\tilde{X}$. Then
$\bar{y},\underline{y}\in\mathcal{V}^1+\mathcal{M}_{loc}$\,,
$\bar{y},\underline{y}$ are of class (D) and
\[
E\int^T_0 f^+(s,\bar{y}_s,0)\,ds+E\int^T_0
f^-(s,\underline{y}_s,0)\,ds<\infty.
\]
Since we already have shown that $\bar{y}_t\le
Y^n_t\le\underline{y}_t$, $t\in[0,T]$, condition (d) is satisfied.
Condition (e) follows from (H3), whereas (f) is satisfied by the
very definition of $Y$. Thus all the assumptions of Theorem
\ref{kl4.12} are satisfied. Therefore  $Y$ is regulated and there
exist $K\in\mathcal{V}^+$, $Z\in\mathcal{H}$ such that
\[
Y_t=\xi+\int^T_tf(s,Y_s,Z_s)\,ds+\int^T_t\,dV_s
+\int^T_t\,dK_s-\int^T_t Z_s\,dB_s,\quad t\in[0,T].
\]
Furthermore,  $Z^n\rightarrow Z$ in measure $\lambda\otimes\,P$,
which when combined with (\ref{Tw.3.1}) and (\ref{Tw.3.2}) implies
that if $p>1$ then $Z\in\mathcal{H}^p$ and (\ref{eq4.6}) is
satisfied, whereas if $p=1$, then $Z\in\mathcal{H}^q$ for
$q\in(0,1)$  and there exists a stationary sequence $\{\tau_k\}$
such that
\begin{equation}\label{Tw3.3}
E\int^{\tau_k}_0|Z^n_s-Z_s|^{\gamma}\,ds\rightarrow 0,\quad \gamma\in[1,2).
\end{equation}
We will show that
\begin{equation}\label{Tw.3.4}
\sup_{n\ge 1}E\Big(\int^T_0|f(s,Y^n_s,Z^n_s)|\,ds\Big)^p
+E\Big(\int^T_0 |f(s,Y_s,Z_s)|\,ds\Big)^p<\infty.
\end{equation}
If  $p>1$ then  by (H1),
\begin{align*}
&E\Big(\int^T_0|f(s,Y^n_s,Z^n_s)|\,ds\Big)^p\\
&\quad\le C_p\Big(\Big(\int^T_0|f(s,\tilde{X}_s,0)|\,ds\Big)^p
+\Big(\int^T_0|f(s,Y^1_s,0)|\,ds\Big)^p
+E\Big(\int^T_0|Z^n_s|^2\,ds\Big)^{p/2}\Big).
\end{align*}
If $p=1$ then  by (Z),
\[
E\int^T_0 |f(s,Y^n_s,Z^n_s)|\,ds \le \gamma E\int^T_0
(g_s+|Y^n_s|+|Z^n_s|)^{\alpha}\,ds +E\int^T_0 |f(s,Y^n_s,0)|\,ds.
\]
By H\"older's  inequality and (H2),
\begin{align*}
&\gamma E\int^T_0 (g_s+|Y^n_s|+|Z^n_s|)^{\alpha}\,ds
+E\int^T_0 |f(s,Y^n_s,0)|\,ds\\
&\quad\le E\Big(\int^T_0|Z^n_s|^2\,ds\Big)^{\alpha/2}
+\gamma E\int^T_0 (g_s+|\tilde{X}_s|+|Y^1_s|)^{\alpha}\,ds\\
&\qquad+E\int^T_0 |f(s,Y^1_s,0)|+|f(s,\tilde{X}_s,0)|\,ds.
\end{align*}
By Fatou's lemma, (\ref{Tw.3.1}), (\ref{Tw.3.2}) we have
(\ref{Tw.3.4}),  which when combined with integrability of $Y,K$
implies that $K\in\mathcal{V}^{p,+}$.

{\em Step 3.} We show that the triple $(Y,Z,K)$ is a solution of
RBSDE($\xi$,$f+dV$,$L$). By (\ref{Tw.3.4}),
$\sup_{n\ge1}EK^n_T<\infty$, so $\{n\int_0^T(Y^n_s-L_s)^-ds\}$ is
bounded in $L^1$. Therefore, up to a subsequence,
$(Y^n_t-L^n_t)^-\to0$ $P$-a.s. for a dense subset of $t$. Hence
$Y_t\geq L_t$ for a dense subset of $t$. Consequently, $Y_{t+}\ge
L_{t+}$ for every $t\in[0,T)$. In fact,  $Y_t\ge L_t$ for every
$t\in[0,T)$. Indeed, if $\Delta^+(L_t+V_t)\ge 0$ for some
$t\in[0,T)$ then
\[
Y_t+V_t=-\Delta^+(Y_t+V_t)+Y_{t+}+V_{t+}\ge Y_{t+}+V_{t+} \ge
L_{t+}+V_{t+}\ge L_t+V_t
\]
whereas if $\Delta^+(L_t+V_t)<0$ for some $t\in[0,T)$ then
$t\in\bigcup_i[[\sigma_{n,i}]]$ for sufficiently large $n$, which
implies that $\Delta^+K^n_t=(Y^n_{t+}-L_t+\Delta^+V_t)^-$. Suppose
that $Y^n_t<L_t$ for some $t$. Since
$\Delta^+(Y_t+V_t)=-\Delta^+K^n_t$, we then have
\begin{align*}
Y^n_{t+}-L_t+\Delta^+V_t<Y^n_{t+}-Y^n_t+\Delta^+V_t
=-(Y^n_{t+}-L_t+\Delta^+V_t)^-,
\end{align*}
which ... contradiction. Thus $Y^n_t\ge L_t$ for every $t\in
[0,T)$, and hence $Y_t\ge L_t$  for $t\in [0,T)$. Consequently,
\[
Y_t\ge L_t\mathbf{1}_{\{t<T\}}+\xi\mathbf{1}_{\{t=T\}},\quad t\in[0,T].
\]
Now we are going to show the minimality condition for $K$. Since
$Y_t+\int^t_0 f(s,Y_s,Z_s)\,ds-V_t$, $t\in[0,T]$, is a
supermartingale, it follows from the properties of the Snell
envelope that
\begin{equation}\label{Tw.3.5}
\begin{split}
Y_t\ge \esssup_{\tau\in\Gamma_t}E\Big(\int^{\tau}_t
f(s,Y_s,Z_s)\,ds+\int^{\tau}_t\,dV_s
+L_{\tau}\mathbf{1}_{\{\tau<T\}}
+\xi\mathbf{1}_{\{\tau=T\}}|\mathcal{F}_t\Big).
\end{split}
\end{equation}
If $p>1$ then by Proposition \ref{prop2.5}  and the definition of
$L^n$, for $t\in[0,T]$ we have
\begin{align*}
Y^n_t&=\esssup_{\tau\in\Gamma_t}E\Big(\int^{\tau}_t
f(s,Y^n_s,Z^n_s)\,ds
+\int^{\tau}_t\,dV_s+L^n_{\tau}\mathbf{1}_{\{\tau<T\}}
+\xi\mathbf{1}_{\{\tau=T\}}|\mathcal{F}_t\Big)\\
&\le \esssup_{\tau\in\Gamma_t}E\Big(\int^{\tau}_t
f(s,Y^n_s,Z^n_s)\,ds +\int^{\tau}_t\,dV_s
+L_{\tau}\mathbf{1}_{\{\tau<T\}}
+\xi\mathbf{1}_{\{\tau=T\}}|\mathcal{F}_t\Big).
\end{align*}
Observe that by  (\ref{eq4.6}), (\ref{Tw.3.4}) and the assumptions
on $f$,
\[
E\int_0^T|f(s,Y^n_s,Z^n_s)-f(s,Y_s,Z_s)|\,ds\rightarrow 0.
\]
By Lemma \ref{lem2},
\begin{align*}
Y_t\le \esssup_{\tau\in\Gamma_t}E\Big(\int^{\tau}_t f(s,Y_s,Z_s)\,ds
+\int^{\tau}_t\,dV_s+L_{\tau}\mathbf{1}_{\tau<T}
+\xi\mathbf{1}_{\tau=T}|\mathcal{F}_t\Big).
\end{align*}
By the above inequality and (\ref{Tw.3.5}),
\[
Y_t=\esssup_{\tau\in\Gamma_t}E\Big(\int^{\tau}_t
f(s,Y_s,Z_s)\,ds+\int^{\tau}_t\,dV_s
+L_{\tau}\mathbf{1}_{\tau<T}+\xi\mathbf{1}_{\tau=T}|\mathcal{F}_t\Big).
\]
By Corollary \ref{wn2.1.0} we have the minimality condition  for
$K$. Hence the triple $(Y,Z,K)$ is a solution of RBSDE($\xi,f+dV,
L$) on $[0,T]$.

Consider now the case  $p=1$. Since $Y^1\le Y^n\le Y$, $n\ge 1$,
by (H2) we have
\[
f(t,Y_t,0)\le f(t,Y^n_t,0)\le f(t,Y^1_t,0),\quad t\in [0,T].
\]
Set
\[
\sigma_k=\inf\{t\ge 0;\, \int_0^t|f(s,Y^1_s,0)|\,ds
+\int_0^t|f(s,Y_s,0)|\,ds\ge k\}\wedge T.
\]
It is clear that $\{\sigma_k\}$ is stationary.  We may assume that
$\sigma_k=\tau_k$. By Proposition \ref{prop2.5} and the definition
of $L^n$,
\begin{align*}
Y^n_t&=\esssup_{\tau_k\ge\tau,\tau\in\Gamma_t}E\Big(\int^{\tau}_t
f(s,Y^n_s,Z^n_s)\,ds+\int^{\tau}_t\,dV_s
+L^n_{\tau}\mathbf{1}_{\{\tau<\tau_k\}}
+Y^n_{\tau_k}\mathbf{1}_{\{\tau=\tau_k\}}|\mathcal{F}_t\Big)\\
&\le \esssup_{\tau_k\ge\tau,\tau\in\Gamma_t} E\Big(\int^{\tau}_t
f(s,Y^n_s,Z^n_s)\,ds+\int^{\tau}_t\,dV_s
+L_{\tau}\mathbf{1}_{\{\tau<\tau_k\}}
+Y_{\tau_k}\mathbf{1}_{\{\tau=\tau_k\}}|\mathcal{F}_t\Big).
\end{align*}
Observe that by  (\ref{Tw3.3}), the definition of  $\sigma_k$ and
the assumptions on $f$,
\[
E\int_0^{\tau_k}|f(s,Y^n_s,Z^n_s)-f(s,Y_s,Z_s)|\,ds\rightarrow 0.
\]
By Lemma \ref{lem2},
\begin{align*}
Y_t\le \esssup_{\tau_k\ge\tau,\tau\in\Gamma_t}E\Big(\int^{\tau}_t
f(s,Y_s,Z_s)\,ds
+\int^{\tau}_t\,dV_s+L_{\tau}\mathbf{1}_{\{\tau<\tau_k\}}
+Y_{\tau_k}\mathbf{1}_{\{\tau=\tau_k\}}|\mathcal{F}_t\Big).
\end{align*}
By the above inequality and  (\ref{Tw.3.5}),
\begin{align*}
Y_t=\esssup_{\tau_k\ge\tau,\tau\in\Gamma_t}E\Big(\int^{\tau}_t
f(s,Y_s,Z_s)\,ds+\int^{\tau}_t\,dV_s
+L_{\tau}\mathbf{1}_{\{\tau<\tau_k\}}
+Y_{\tau_k}\mathbf{1}_{\{\tau=\tau_k\}}|\mathcal{F}_t\Big).
\end{align*}
By Corollary \ref{wn2.1.0} we have the minimality  condition for
$K$ on $[0,\tau_k]$, and by stationarity of $\{\tau_k\}$ also on
$[0,T]$. Therefore  $(Y,Z,K)$ is the solution of
RBSDE$(\xi,f+dV,L)$ on $[0,T]$.
\end{proof}

\nsubsection{Appendix. It\^o's formula for processes with
regulated trajectories}

We  consider an $\mathbb{F}$-adapted process $X$ with regulated
trajectories of the form
\begin{equation}\label{eq5.12}
X_t=X^*_t+\sum_{s<t}\Delta^+X_s,\quad t\in[0,T],
\end{equation}
where $X^*$  is an $\mathbb{F}$-adapted semimartingale with
c\`adl\`ag trajectories  and
\[
\sum_{s<T}|\Delta^+X_s|<\infty,\quad P\mbox{-a.s.}
\]
(note that $\Delta^-X_s=\Delta X^*_s$).

\begin{theorem}[\cite{Ga,Le}]
\label{thm1} Let $(X_t)_{t\leq T}$ be an adapted process with
regulated trajectories  of the form \mbox{\rm(\ref{eq5.12})}, and
let $f$ be a real function of class $C^2$. Then the process
$(f(X_t))_{t\leq T}$ also has the form  \mbox{\rm(\ref{eq5.12})}.
More precisely, for every $t\in[0,T]$,
\[
f(X_t)=f(X_0)+ \int_0^tf'(X_{s-})\,dX^*_s
+\frac12\int_0^tf''(X_{s-})\,d[X^*]_s^c+J^-_t+J^+_t,\] where
$\displaystyle{J^-_t=\sum_{s\leq
t}\{f(X_{s})-f(X_{s-})-f'(X_{s-})\Delta^- X_s\}}$,
$\displaystyle{J^+_t=\sum_{s<t}\{f(X_{s+})-f(X_s)\}}$.
\end{theorem}

Note that the two sums defining $J^-$  and $J^+$ are absolutely
convergent, and that $J^-$ is a c\`adl\`ag adapted process,
whereas $J^+$ is c\`agl\`ad adapted.  Indeed,
\[
|J^-|_t\leq C_1\sum_{s\leq t}|\Delta^-X_s|^2=C_1\sum_{s\leq
t}|\Delta X^*_s|^2, \quad P\mbox{-a.s.}
\]
and
\[
|J^+|_t\leq C_2\sum_{s<t}|\Delta^+X_s|,\quad P\mbox{-a.s.},
\]
where $C_1,C_2$  are random variables defined by
$C_1=(1/2)\sup_{x\in[-M,M]}|f''(x)|$ and
$C_2=\sup_{x\in[-M,M]}|f'(x)|$, where $M=\sup_{s\leq T}|X_s|$
(note that $M <\infty$ $P$-a.s.) We  include the proof of Theorem
\ref{thm1} for completeness of our presentation.
\begin{proof} Set  $X^+_t=X_{t+}$, $t\leq T$. Clearly
\[
X^+_t=\Delta^+X_t+X_t=X^*_t+\sum_{s\leq t}\Delta^+X_s,\quad t\leq T.
\]
Hence $X^+$ is a semimartingale . By It\^o's formula for
semimartingales,
\begin{align*}
f(X^+_t)&=f(X_0)+
\int_0^tf'(X^+_{s-})\,dX^+_s+\frac12\int_0^tf''(X^+_{s-})\,d[X^*]_s^c\\
&\quad+\sum_{s\leq t} \{f(X^+_{s})-f(X^+_{s-})-f'(X^+_{s-})\Delta
X^+_s\}.
\end{align*}
Observe that $X^+_{s-}=X_{s-}$\,,
$f(X^+_s)=f(X_s)+f(X_{s+})-f(X_s)$ and $\Delta
X^+_s=\Delta^+X_s+\Delta^-X_s$. Hence
 \begin{align}
 \label{eq5.2}
f(X^+_t)&=f(X_0)+
\int_0^tf'(X_{s-})\,dX^*_s+\sum_{s\leq t}f'(X_{s-})\Delta^+X_s
+\frac12\int_0^tf''(X^+_{s-})\,d[X^*]_s^c\nonumber\\
&\quad+\sum_{s\leq t}
\{f(X_{s+})-f(X_{s-})-f'(X_{s-})(\Delta^+ X_s+\Delta^-X_s)\}\nonumber\\
&=f(X_0)+\int_0^tf'(X_{s-})\,dX^*_s +\frac12\int_0^tf''(X^+_{s-})\,d[X^*]_s^c
\nonumber\\
&\quad+\sum_{s\leq t}
\{f(X_{s})-f(X_{s-})-f'(X_{s-})\Delta^-X_s\}+\sum_{s\leq t}
\{f(X_{s+})-f(X_{s})\}.
\end{align}
Subtracting $f(X_{t+})-f(X_{t})$ from both sides of (\ref{eq5.2})
we obtain the desired formula.
\end{proof}

\begin{corollary}\label{cor1}
Let $X=(X^1,\dots,X^d)$ be an adapted $d$-dimensional process with
regulated trajectories of the form \mbox{\rm(\ref{eq5.12})} and
let $f:\Rd\to \R$ is a function of class $C^2$. Then the process
$(f(X_t))_{t\leq T}$ also has the form  \mbox{\rm(\ref{eq5.12})}.
Moreover, for every $t\in[0,T]$,
\begin{align*}
f(X_t)&=f(X_0)+
\sum_{i=1}^d\int_0^t\frac{\partial f}{\partial x_i}(X_{s-})\,dX^{i,*}_s\\
&\quad+\frac12\sum_{i=1}^d\sum_{j=1}^d\int_0^t\frac{\partial^2
f}{\partial x_i\partial
x_j}(X_{s-})\,d[X^{i,*},X^{j,*}]_s^c+J^-_t+J^+_t,\end{align*}
where $\displaystyle{J^-_t=\sum_{s\leq
t}\{f(X_{s})-f(X_{s-})-\sum_{i=1}^d\frac{\partial f}{\partial
x_i}(X_{s-})\Delta^- X^i_s\}}$,
$\displaystyle{J^+_t=\sum_{s<t}\{f(X_{s+})-f(X_s)\}}$.
\end{corollary}
\begin{corollary}\label{cor2}
Let  $X^1, X^2$ be two adapted processes with regulated
trajectories of the form \mbox{\rm(\ref{eq5.12})}. Then
\begin{align*}
X^1_tX^2_t&=X^1_0X^2_0+\int_0^tX^1_{s-}\,dX^{2,*}_s
+\int_0^t X^2_{s-}\,dX^{1,*}_s+[X^{1,*},X^{2,*}]_t\\
&\quad+\sum_{s<t}(X^1_{s+}X^2_{s+}-X^1_sX^2_s),\quad t\in[0,T].
\end{align*}
\end{corollary}

\begin{corollary}\label{cor3}
Let $X=(X^1,\dots,X^d)$ be an adapted  $d$-dimensional process
with regulated trajectories of the form \mbox{\rm(\ref{eq5.12})}.
Then for all $p\geq1$  and  $t\in[0,T]$,
\begin{align*}|X_t|^p=&|X_0|^p+
p\int_0^t|X_{s-}|^{p-1}\langle \hat\sgn( X_{s-}),\,dX^{*}_s\rangle
+p\sum_{s<t}|X_{s}|^{p-1}\langle \hat\sgn(X_{s}),\,\Delta^+ X_s\rangle\\
&+\frac{p}{2}\int_0^t |X_{s}|^{p-2}{\bf 1}_{\{X_{s}\neq0\}}\{(2-p)
|(1-\langle \hat\sgn(X_{s}),Q^{X^{*}}_s
\hat\sgn(X_{s})\rangle)+(p-1)\} d[X^{*}]_s^c\\
&+L_t{\bf 1}_{\{p=1\}} + J^-_t(p)+J^+_t(p),
\end{align*} where
$Q^{X^{*}}$ denotes the Radon-Nikodym derivative
$d[[X^\star]]^c/d[X^\star]^c$, $(L_t)_{t\leq T}$  is an adapted
increasing continuous process such that $L_0=0$, and
\[J^-_t(p)=\sum_{s\leq
t}\{|X_{s}|^p-|X_{s-}|^p-p|X_{s-}|^{p-1}\langle
\hat\sgn(X_{s-}),\,\Delta^- X_s\rangle\},\quad t\in[0,T]
\]
and
\[
J^+_t(p)=\sum_{s< t}
\{|X_{s+}|^p-|X_{s}|^p-p|X_{s}|^{p-1}\langle\hat\sgn(X_{s}),\,\Delta^+
X_s\rangle\},\quad t\in[0,T]
\]
are adapted increasing processes with c\`adl\`ag and c\`agl\`ad
trajectories, respectively.
\end{corollary}
\begin{proof} We follow the proof of \cite[Lemma 2.2]{bdh}
(see also the proof of \cite[Proposition 2.1]{kl}).  The  formula
is an easy consequence of Corollary \ref{cor1} in the case where
$p\geq2$. Assume that $p\in[1,2)$ and for  $\epsilon>0$ set
$u_\epsilon(x)=(|x|^2+\epsilon^2)^{1/2}$, $x\in\Rd$. Clearly,
$u_\epsilon^p$ is a smooth approximation of $|\cdot|^p$. It is
easy to check that $\frac{\partial u_\epsilon^p}{\partial
x_i}(x)=pu_\epsilon^{p-2}(x)x_i$ for $i=1,\dots,d$, $x\in\Rd$,
and
\[
\frac{\partial^2 u^p_\epsilon}{\partial x_i\partial x_j}(x)
=p(p-2)u_\epsilon^{p-4}(x)x_ix_j+pu_\epsilon^{p-2}(x){\bf
1}_{\{i=j\}}, \quad i,j=1,\dots,d,\,x\in\Rd.\] By Corollary
\ref{cor1},
\begin{align*}
u_\epsilon^p(X_t)&=u_\epsilon^p(X_0)
+p\int_0^tu_\epsilon^{p-2}(X_{s-})\langle X_{s-},dX^{*}\rangle
+p\sum_{s<t}u_\epsilon^{p-2}(X_{s})\langle X_s,\Delta^+X_s\rangle\\
&\quad+\frac12\sum_{i=1}^d\sum_{j=1}^d
\int_0^t\{p(p-2)u^{p-4}_\epsilon(X_s)X^i_sX^j_s
+pu_\epsilon^{p-2}(X_s){\bf 1}_{\{i=j\}}\}d[X^{i,*},X^{j,*}]^c_s\\
&\quad+\sum_{s\leq t}
\{u^p_\epsilon(X_{s})-u^p_\epsilon(X_{s-})-pu_\epsilon^{p-2}(X_{s-})\langle
X_{s-}, \Delta^- X_s\rangle\}\\
&\quad+\sum_{s< t}
\{u^p_\epsilon(X_{s+})-u^p_\epsilon(X_{s})-pu_\epsilon^{p-2}(X_{s})\langle
X_s, \Delta^+X_s\rangle\}\\
&=:u_\epsilon^p(X_0)+I^{1,\epsilon}_t+I^{2,\epsilon}_t+I^{3,\epsilon}_t
+I^{4,\epsilon}_t+I^{5,\epsilon}_t,
\end{align*}
where using (\ref{eq5.12}) we separated $I^{2,\epsilon}$ from
right side jumps $J^+$. Since $u^p_\epsilon(x)\to|x|^p$ $x\in\Rd$,
it is clear that
\begin{equation}
\label{eq5.4} u^p_\epsilon(X_t)\rightarrow|X_t|^p,\quad t\in[0,T],
\quad P\mbox{-a.s.}
\end{equation}
Moreover, the convergence
$u_\epsilon^{p-2}(x)x\to|x|^{p-1}\hat\sgn(x)$, $x\in\Rd$ implies
that
\begin{equation}\label{eq5.5}
I^{1,\epsilon}_t\arrowp p\int_0^t|X_{s-}|^{p-1}\langle \hat\sgn(
X_{s-}),\,dX^{*}_s\rangle,\quad t\in[0,T]\end{equation} and, by
(\ref{eq5.12}), that
\begin{equation}\label{eq5.6}
I^{2,\epsilon}_t\rightarrow p\sum_{s<t}|X_{s-}|^{p-1}\langle
\hat\sgn( X_{s}),\Delta^+X_s\rangle,\quad t\in[0,T],\quad
P\mbox{-a.s.}
\end{equation}
Similarly,
\begin{equation}\label{eq5.7}
I^{5\epsilon}_t\rightarrow J^+_t,\quad t\in[0,T],\quad
P\mbox{-a.s.}
\end{equation} On the other hand,
using the identity
$u_\epsilon^{p-2}(x)=u^{p-4}_\epsilon(x)|x|^2+\epsilon^2u_\epsilon^{p-4}(x)$
we get
\begin{align*}
I^{3,\epsilon}_t&=\frac12\sum_{i=1}^d\sum_{j=1}^d
\int_0^t\{p(p-2)u^{p-4}_\epsilon(X_s)X^i_sX^j_s
+pu_\epsilon^{p-4}(X_s)|X_s|^2{\bf 1}_{\{i=j\}}\}d[X^{i,*},X^{j,*}]^c_s\\
&\quad+\frac12\sum_{i=1}^d\sum_{j=1}^d
\int_0^tp\epsilon^2u_\epsilon^{p-4}(X_s){\bf 1}_{\{i=j\}}\}d[X^{i,*},X^{j,*}]^c_s\\
&=\frac12p\sum_{i=1}^d\sum_{j=1}^d
\int_0^t(2-p)u^{p-4}_\epsilon(X_s)|X_s|^2
\big({\bf 1}_{\{i=j\}}-\frac{X^i_s}{|X_s|}
\frac{X^j_s}{|X_s|}\big){\bf 1}_{\{X_s\neq0\}}d[X^{i,*},X^{j,*}]^c_s\\
&\quad+\frac12p\sum_{i=1}^d
\int_0^t(p-1)u_\epsilon^{p-4}(X_s)|X_s|^2d[X^{i,*}]^c_s+
\frac{p}{2}\sum_{i=1}^d\int_0^t
\epsilon^2u_\epsilon^{p-4}(X_s)d[X^{i,*}]^c_s\\
&=\frac{p}{2}\sum_{i=1}^d\sum_{j=1}^d
\int_0^t(2-p)u^{p-4}_\epsilon(X_s)|X_s|^2\big({\bf 1}_{\{i=j\}}
-\frac{X^i_s}{|X_s|}\frac{X^j_s}{|X_s|}\big)
Q^{X^\star}_s(i,j){\bf 1}_{\{X_s\neq0\}}d[X^\star]^c_s\\
&\quad+\frac{p}{2}\int_0^t(p-1)u_\epsilon^{p-4}(X_s)|X_s|^2d[X^{*}]^c_s+
\frac{p}{2}\int_0^t\epsilon^2u_\epsilon^{p-4}(X_s)d[X^{*}]^c_s\\
&=:I^{6,\epsilon}_t+I^{7,\epsilon}_t+I^{8,\epsilon}_t.
\end{align*}
Since $Q^{X^\star}_s$  is a symmetric non-negative matrix with  a trace equal to 1,
\begin{align}
\nonumber\sum_{i=1}^d\sum_{j=1}^d&\big({\bf 1}_{\{i=j\}}
-\frac{X^i_s}{|X_s|}\frac{X^j_s}{|X_s|}\big)
Q^{X^\star}_s(i,j){\bf 1}_{\{X_s\neq0\}}\\
&=(1-\langle\hat\sgn(X_s),Q^{X^\star}_s\hat\sgn(X_s)\rangle){\bf
1}_{\{X_s\neq0\}}\geq0,\quad s\in[0,T].\label{eq5.8}
\end{align}
By this and the fact that $|x|/u_\epsilon(x)\nearrow{\bf
1}_{\{x\neq0\}}$, $x\in\Rd$, it follows that for $t\in[0,T]$,
\begin{equation}
\label{eq5.9}
I^{6,\epsilon}_t\nearrow\frac12p\int_0^t(2-p)|X_s|^{p-2}
(1-\langle\hat\sgn(X_s),Q^{X^\star}_s\hat\sgn(X_s)\rangle){\bf
1}_{\{X_s\neq0\}}d[X^\star]_s^c
\end{equation}
$P$-a.s. Similarly,
\begin{equation}\label{eq5.10}
I^{7,\epsilon}_t\nearrow\frac12p\int_0^t(p-1)|X_s|^{p-2}{\bf
1}_{\{X_s\neq0\}}d[X^\star]_s^c,\quad t\in[0,T],\quad
P\mbox{-a.s.}
\end{equation}
From (\ref{eq5.4})--(\ref{eq5.10}) we deduce that there is a
process $B$  with regulated trajectories such that
$I^{4,\epsilon}_t+I^{8,\epsilon}_t\rightarrow B_t$ in probability
$P$ for $t\in[0,T]$, and
\begin{align}
\label{eq5.11} |X_t|^p&=|X_0|^p+ p\int_0^t|X_{s-}|^{p-1}\langle
\hat\sgn( X_{s-}),\,dX^{*}_s\rangle
+p\sum_{s<t}|X_{s}|^{p-1}\langle
\hat\sgn(X_{s}),\,\Delta^+ X_s\rangle\nonumber\\
&+\frac12 p\int_0^t |X_{s}|^{p-2}{\bf 1}_{\{X_{s}\neq0\}}
\{(2-p) |(1-\langle \hat\sgn(X_{s}),Q^{X^{*}}_s
\hat\sgn(X_{s})\rangle)+(p-1)\} d[X^{*}]_s^c\nonumber\\
&+B_t +J^+_t(p).
\end{align}
Since the function $u^p_\epsilon$ is convex, the processes
$I^{8,\epsilon}, I^{4,\epsilon}$ are increasing. It follows that
$B$ is also increasing. Moreover, $B_0=0$ and $B_t=L_t+\sum_{s\leq
t}\Delta^-B_s+\sum_{s<t}\Delta^+B_s$, where $L$ is the continuous
part of $B$. Comparing the jumps of the left and right-hand side
of (\ref{eq5.11}) we obtain that $\sum_{s\leq
t}\Delta^-B_s=J^-_t(p)$ and $\sum_{s<t}\Delta^+B_s=0$. Moreover,
it follows from the arguments from the proof of \cite[Lemma
2.2]{bdh} that $L=0$ in the case where $p>1$.
\end{proof}
\begin{corollary}\label{cor4} Let $X=(X^1,\dots,X^d)$ be an adapted
$d$-dimensional process with regulated trajectories of the form
\mbox{\rm(\ref{eq5.12})}. Then for all $p\in[1,2]$  and
$t\in[0,T]$,
\begin{align*}
|X_t|^p&+\frac{p(p-1)}2\int_0^t |X_{s}|^{p-2}{\bf
1}_{\{X_{s}\neq0\}} d[X^{*}]_s^c +J^-_T(p)-J^-_t(p)
+J^+_T(p)-J^+_t(p)\\
&\leq
 |X_T|^p+ p\int_t^T|X_{s-}|^{p-1}\langle
\hat\sgn( X_{s-}),\,dX^{*}_s\rangle +p\sum_{t\leq
s<T}|X_{s}|^{p-1}\langle \hat\sgn(X_{s}),\,\Delta^+
X_s\rangle.\end{align*}
\end{corollary}
\begin{proof} Follows from Corollary \ref{cor3} and
(\ref{eq5.8}).
\end{proof}

\noindent{\bf Acknowledgements}\\
Research supported by  NCN grant no. 2012/07/B/ST1/03508.


\end{document}